\documentclass[leqno]{article} 

\usepackage[english]{babel}
\usepackage[T1]{fontenc}

\usepackage{amsmath}
\usepackage{amssymb}
\usepackage{amsfonts}
\usepackage{enumerate}
\usepackage{vmargin}
\usepackage[all]{xy}
\usepackage{mathrsfs}
\usepackage{mathtools}
\usepackage{lmodern}
\usepackage[pagebackref=true]{hyperref}

\setmarginsrb{3cm}{3cm}{3.5cm}{3cm}{0cm}{0cm}{1.5cm}{3cm}

\newcommand{\C}{\ensuremath{\mathbb{C}}}

\newcommand{\E}{\ensuremath{\mathbb{E}}}

\newcommand{\N}{\ensuremath{\mathbb{N}}}
\newcommand{\R}{\ensuremath{\mathbb{R}}}

\newcommand{\Z}{\ensuremath{\mathbb{Z}}}

\renewcommand{\leq}{\ensuremath{\leqslant}}
\renewcommand{\geq}{\ensuremath{\geqslant}}
\newcommand{\qed}{\hfill \vrule height6pt  width6pt depth0pt}

\newcommand{\norm}[1]{\left\Vert#1\right\Vert}

\newcommand{\ot}{\otimes}
\newcommand{\epsi}{\varepsilon}

\newcommand{\supp}{\mathrm{supp}}

\newcommand{\dist}{\mathrm{dist}}
\newcommand{\Id}{\mathrm{Id}}

\newcommand{\HI}{H^\infty}
\newcommand{\Hor}{\mathcal{H}}
\newcommand{\type}{\mathrm{type}\,}
\newcommand{\cotype}{\mathrm{cotype}\,}
\newcommand{\dyad}{\varphi}
\newcommand{\cutoff}{\chi}
\newtheorem{thm}{Theorem}[section]
\newtheorem{defi}[thm]{Definition}

\newtheorem{prop}[thm]{Proposition}

\newtheorem{cor}[thm]{Corollary}
\newtheorem{lemma}[thm]{Lemma}

\newtheorem{remark}[thm]{Remark}

\newtheorem{ass}[thm]{Assumption}

\newenvironment{proof}[1][]{\noindent {\it Proof #1} : }{\hbox{~}\qed
\smallskip
}

\numberwithin{equation}{section}

\begin{document}
\selectlanguage{english}
\title{\bfseries{$q$-variational H\"ormander functional calculus and Schr\"odinger and wave maximal estimates}}
\date{March 2024}
\author{\bfseries{Luc Deleaval and Christoph Kriegler}}

\maketitle

\begin{abstract}
This article is the continuation of the work \cite{DK} where we had proved maximal estimates 
\[ \norm{ \sup_{t > 0} |m(tA)f| \: }_{L^p(\Omega,Y)} \leq C \norm{f}_{L^p(\Omega,Y)} \]
for sectorial operators $A$ acting on $L^p(\Omega,Y)$ ($Y$ being a UMD lattice) and admitting a H\"ormander functional calculus
(a strengthening of the holomorphic $\HI$ calculus to symbols $m$ differentiable on $(0,\infty)$ in a quantified manner), and $m : (0, \infty) \to \C$ being a H\"ormander class symbol with certain decay at $\infty$.
In the present article, we show that under the same conditions as above, the scalar function $t \mapsto m(tA)f(x,\omega)$ is of finite $q$-variation with $q > 2$, a.e. $(x,\omega)$.
This extends recent works by \cite{BMSW,HHL,HoMa1,HoMa,JSW,LMX} who have considered among others $m(tA) = e^{-tA}$ the semigroup generated by $-A$.
As a consequence, we extend estimates for spherical means in euclidean space from \cite{JSW} to the case of UMD lattice-valued spaces.
A second main result yields a maximal estimate 
\[ \norm{ \sup_{t > 0} |m(tA) f_t| \: }_{L^p(\Omega,Y)} \leq C \norm{f_t}_{L^p(\Omega,Y(\Lambda^\beta))} \]
for the same $A$ and similar conditions on $m$ as above but with $f_t$ depending itself on $t$ such that $t \mapsto f_t(x,\omega)$ belongs to a Sobolev space $\Lambda^\beta$ over $(\R_+, \frac{dt}{t})$.
We apply this to show a maximal estimate of the Schr\"odinger (case $A = -\Delta$) or wave (case $A = \sqrt{-\Delta}$) solution propagator $t \mapsto \exp(itA)f$.
Then we deduce from it variants of Carleson's problem of pointwise convergence \cite{Car}
\[ \exp(itA)f(x,\omega) \to f(x,\omega) \text{ a. e. }(x,\omega) \quad (t \to 0+)\]
for $A$ a Fourier multiplier operator or a differential operator on an open domain $\Omega \subseteq \R^d$ with boundary conditions.
\end{abstract}


\makeatletter
 \renewcommand{\@makefntext}[1]{#1}
 \makeatother
 \footnotetext{
 {\it Mathematics subject classification:}
 42A45, 42B25, 47A60.
\\
{\it Key words}: Spectral multiplier theorems, UMD valued $L^p$ spaces, maximal estimates.}

 \tableofcontents

\section{Introduction}
\label{sec-Introduction}

This article is a follow-up of our work \cite{DK} concerning maximal estimates for H\"ormander spectral multipliers on $L^p$ spaces and UMD lattices.
Let us recall the setting in the context of the euclidean Laplacian which provides in fact the leading example for what follows.
The operator $-\Delta$ (defined e.g. as Fourier multiplier with symbol $|\xi|^2$) is a self-adjoint operator on $L^2(\R^d)$.
As such, it has a functional calculus allowing to insert it into a bounded Borel function $m : \R_+ \to \C$ and to obtain a bounded operator $m(-\Delta)$ on $L^2(\R^d)$.
Then one can ask for which functions, the operator $m(-\Delta)$ is bounded on $L^p(\R^d)$ for fixed $1 < p < \infty$.
H\"ormander's theorem \cite[Theorem 2.5]{Hor}, based on Mihlin's work \cite{Mi56} gives the sufficient condition
\begin{equation}
\label{equ-intro-classical-Hormander}
\|m\|_{\Hor^\alpha_2}^2 := \max_{k=0,1,\ldots,\alpha} \sup_{R > 0} \frac{1}{R}\int_{R}^{2R} \Bigl| t^k \frac{d^k}{dt^k} m(t) \Bigr|^2 \,dt < \infty,
\end{equation}
where $\alpha$ is an integer strictly larger than $\frac{d}{2}$.
Later on, H\"ormander's theorem has been strengthened to the following maximal estimate:
\begin{equation}
\label{equ-1-intro-maximal-Hor}
\left\| \sup_{t > 0} |m(-t \Delta) f| \, \right\|_{L^p(\R^d)} \leq C \|f\|_{L^p(\R^d)}.
\end{equation}
This cannot hold true for all functions $m$ with $\norm{m}_{\Hor^\alpha_2} < \infty$ \cite{CGHS} even for large values of $\alpha$, so that one has to restrict to a subclass of spectral multipliers for which \eqref{equ-1-intro-maximal-Hor} holds.
Admissible spectral multipliers typically have to have a somewhat higher order of differentiability than $\alpha > \frac{d}{2}$ and to decay at $\infty$ in a prescribed order.
Early results for the euclidean Laplacian are due to \cite{Crb,DaTr,RdF2,See}.

H\"ormander's classical and paramount result \eqref{equ-intro-classical-Hormander} of Fourier multipliers has found many generalizations over the last 60 years.
First, we will use in this article also a version for non-integer $\alpha$, in which \eqref{equ-intro-classical-Hormander} is replaced by a Sobolev norm.
Second, the $L^p$-boundedness question makes literally sense for any self-adjoint operator $A$, and in fact,
a theorem of H\"ormander type holds true for many elliptic differential operators $A,$ including sub-Laplacians on Lie groups of polynomial growth, Schr\"odinger operators and elliptic operators on Riemannian manifolds, see \cite{Alex,Christ,Duong,DuOS}.
More recently, spectral multipliers have been studied for operators acting on $L^p(\Omega)$ only for a strict subset of $(1,\infty)$ of exponents $p$ \cite{Bl,CDY,CO,COSY,KuUhl,KU2,SYY}.
Third, another generalization of spectral multipliers is to bring a UMD Banach space $Y$ into the picture and ask for the boundedness of $m(A) \ot \Id_Y$ on the Bochner space $L^p(\Omega,Y)$.
In the case of euclidean Laplacian $A = - \Delta$, this has been developed by Hyt\"onen \cite{Hy1,Hy2} and Girardi and Weis \cite{GiWe}, and in the case of self-adjoint semigroups admitting (generalised) Gaussian estimates, with $Y$ a UMD lattice, by Kemppainen and the authors \cite{DKK}.
Fourth, in view of \eqref{equ-1-intro-maximal-Hor}, one can ask whether the following maximal estimate holds true:
\begin{equation}
\label{equ-2-intro-maximal-Hor}
\left\| \sup_{t > 0} |m(tA) f| \, \right\|_{L^p(\Omega,Y)} \leq C \|f\|_{L^p(\Omega,Y)},
\end{equation}
and for which functions $m : \R_+ \to \C$.

In this work as in its predecessor \cite{DK}, we consider $A$ to fit into the following quite general framework of semigroup generators, which captures all the situations above.
That is, we assume that $A$ generates a $c_0$-semigroup on $L^p(\Omega,Y)$ where $Y$ is a UMD Banach lattice (which englobes the two important particular cases $Y = \C$ and $A = A_0 \ot \Id_Y$ for some semigroup generator $A_0$ on $L^p(\Omega)$, for which the results in this article are already new).
Moreover, we assume the semigroup to be analytic on $\C_+$, so that $A$ is assumed $0$-sectorial, thus allowing the machinery of $\HI(\Sigma_\omega)$ calculus for any angle $\omega \in (0,\pi)$.
Finally, we assume that this $\HI(\Sigma_\omega)$ calculus of $A$ extends a priori to a $\Hor^\alpha_2$ calculus on $L^p(\Omega,Y)$.
All these assumptions are satisfied for (differential) operators on doubling manifolds that are self-adjoint and have (generalised) Gaussian estimates, with $Y = \C$ and also with $A = A_0 \ot \Id_Y$ for general UMD lattices $Y$ in case that $(T_t)_t$ is lattice positive or regular contractive.
This englobes:
\begin{enumerate}
\item The heat semigroup on a complete Riemannian manifold with non-negative Ricci curvature \cite{LY}, \cite[p. 3/70 (1.3)]{GriTel}, \cite{Sal}, \cite[Theorem 4.2.1 \& p.~45]{Fen}, \cite[Proposition 4.8]{DKK}.
\item Schr\"odinger operators on connected and complete Riemannian manifolds with non-negative Ricci curvature and locally integrable, positive potential \cite[Section 7.4, (7.8)]{DuOS}, \cite[Proposition 4.8]{DKK}.
\item Other Schr\"odinger and elliptic differential operators acting on $L^2(\Omega)$, where $\Omega \subseteq \R^d$ is an open subset of homogeneous type \cite{Ouh06}, \cite[Section 6.4, in particular Theorems 6.10, 6.11]{Ouh}.
\item Sub-laplacians on Lie groups with polynomial volume growth \cite[Theorem 4.2, Example 2]{Sal}, \cite{Gri}, \cite[Corollary 4.9]{DKK}.
\item Heat semigroups on fractals \cite[(1.4)]{GriTel}, \cite[Section 7.11]{DuOS}.
\item
For a discussion of many further examples, we refer to \cite[Section 7]{DuOS}, see also \cite[Subsection 5.1]{DKK}.
\end{enumerate}

We refer to \cite[Section 4]{DK} for further examples of operators $A$ and to Section \ref{sec-examples} for relevant UMD lattices other than $L^p$ spaces.
This section contains also an example of a non-selfadjoint  operator $A$ which still has a $\Hor^\alpha_2$ calculus on $L^2(\Omega)$ and on $L^p(\Omega)$.
Then the following main result from \cite[Theorem 1.1]{DK} gives a sufficient condition for \eqref{equ-2-intro-maximal-Hor}.
\begin{thm}[\cite{DK}]
\label{thm-intro-thm11-DK}
Let $Y = Y(\Omega')$ be a UMD lattice, $1 < p < \infty$ and $(\Omega,\mu)$ a $\sigma$-finite measure space.
Let $A$ be a $0$-sectorial operator on $L^p(Y)$.
Assume that $A$ has a $\Hor^\alpha_2$ calculus on $L^p(Y)$ for some $\alpha > \frac12$.
Let $m \in W^{c}_2(\R)$ be a spectral multiplier with $m(0) = 0$ and $c > \alpha + \max\left(\frac12,\frac{1}{\type L^p(Y)} - \frac{1}{\cotype L^p(Y)}\right) + 1$ such that
\begin{equation}
\label{equ-1-intro-thm-main}
\sum_{n \in \Z} \|m(2^n \cdot) \dyad_0 \|_{W^c_2(\R)} < \infty
\end{equation}
for some $\dyad_0 \in C^\infty_c(0,\infty)$ with $\dyad_0(t) = 1$ for $t \in (1,2)$.
Then for a.e. $(x,\omega) \in \Omega \times \Omega'$, $t \mapsto m(tA)f(x,\omega)$ belongs to $C_0(\R_+)$ and
\begin{equation}
\label{equ-2-intro-thm-main}
\| \sup_{t > 0} |m(tA)f|\,\|_{L^p(Y)} \lesssim  \sum_{n \in \Z} \|m(2^n \cdot) \dyad_0 \|_{W^c_2(\R)}\|f\|_{L^p(Y)}.
\end{equation}
\end{thm}

In this article, we extend the above maximal estimate in two ways.
First, we replace the supremum by the $q$-variation norm.
\begin{defi}
\label{def-intro-q-variation}
Let $q \in [1 ,\infty]$.
For a function $a : \R_+ \to \C$, we define the $q$-variation
\[ \|a\|_{V^q} = \sup \left\{ \left( |a_{t_0}|^q + \sum_{k = 1}^\infty |a_{t_k} - a_{t_{k-1}}|^q \right)^{\frac1q} \right\} , \]
where the supremum runs over all increasing sequences $(t_k)_{k \in \N}$ in $\R_+$.
In case of $q = \infty$, the $\infty$-variation is the usual maximal function.
Then the space $V^q$ consisting of all functions with finite $q$-variation is a Banach space \cite[Section 1]{LMX}.
\end{defi}

Note that $V^q \hookrightarrow L^\infty(\R_+)$ \cite[Proposition 5]{BLCS}.
Then our first main result reads as follows.

\begin{thm}[See Theorem \ref{thm-q-variation}]
\label{thm-intro-q-variation}
Assume that $A$ and $m$ satisfy the hypotheses of Theorem \ref{thm-intro-thm11-DK} above.
Let $q > 2$.
Then
\[ \left\|\left\| t \mapsto m(tA)f\right\|_{V^q} \right\|_{L^p(Y)} \lesssim \|f\|_{L^p(Y)} .\]
\end{thm}

This theorem provides an abstract framework for results for the euclidean Laplacian and Laplace type operators, and shows a $q$-variational estimate only out of H\"ormander calculus.
This seems to the best of our knowledge to be the first result on $q$-variation of general spectral multipliers, even in the case $A = -\Delta$.
We refer to Theorem \ref{thm-q-variation} for the more general statement allowing for $m(0) \neq 0$.
We hereby generalise recent results of \cite{LMX} (in case that $A =A_0$ generates a positive contractive semigroup on $L^p(\Omega,\C)$), \cite{HoMa1,HoMa} (in case that $A = A_0 \ot \Id_Y$ on $L^p(\Omega,Y)$ generates an analytic contractively regular semigroup) and of \cite{BMSW,HHL} (in case that $A = (-\Delta)^{\frac12}$ is the square root of the euclidean laplacian), from the specific function $m(\lambda) = e^{-\lambda}$ to other spectral multipliers.

One such spectral multiplier is $m(\lambda) = \frac{2\pi}{|\lambda|^{\frac{d-2}{2}}} J_{\frac{d-2}{2}}(2 \pi \lambda)$ (see \eqref{equ-spectral-multiplier-spherical-mean}) that represents in the case of the euclidean Laplacian $A = - \Delta$ on $L^p(\R^d,Y)$ the spherical mean 
\[ m(-t\Delta)f = A_t f =  \frac{1}{|S_d|} \int_{S^d} f(x-ty) d\sigma(y) , \]
where $S_d = \{ y \in \R^d : \: |y| = 1 \}$ is the euclidean unit sphere equipped with surface measure $\sigma$.
We then obtain, hereby extending the result \cite[Theorem 1.4(i)]{JSW} from the scalar case $Y = \C$ to the UMD lattice case:

\begin{thm}[see Theorem \ref{thm-spherical-means}]
\label{thm-intro-variation-spherical}
Let $1 < p < \infty$ and $2 < q < \infty$.
There exists a minimal dimension $d_0 \in \N$ depending on $p$ and the geometry of $Y$ such that for any $d > d_0$, there is $C_{p,q,d,Y} < \infty$ such that for any $f \in S(\R^d) \otimes Y$, we have
\begin{equation}
\label{equ-intro-variation-spherical}
 \| \| t \mapsto A_t f\|_{V^q} \|_{L^p(\R^d,Y)} \leq C_{p,q,d,Y} \|f\|_{L^p(\R^d,Y)}.
\end{equation}
\end{thm}

Note that a certain minimal dimension in \eqref{equ-intro-variation-spherical} is known to be necessary even in the scalar case \cite[Theorem 1.4(i)]{JSW}.
Then by the method of rotation, we partly recover also a very recent estimate from \cite[Theorem 1.1]{HHL}, namely
\begin{equation}
\label{equ-intro-variation-ball}
\| \| t \mapsto M_t f \|_{V^q} \|_{L^p(\R^d,Y)} \leq C_{p,q,Y} \|f\|_{L^p(\R^d,Y)}
\end{equation}
for the usual euclidean ball mean $M_tf(x) = \frac{1}{|B_d|} \int_{B_d} f(x-ty) dy$ and from the same dimension $d > d_0$ on (see Proposition \ref{prop-HHL-alternative}).
The feature here is that we obtain a dimension free constant $C_{p,q,Y}$.

Second, we return to the maximal estimate as in \eqref{equ-2-intro-thm-main}, but allow the element $f$ to which the operator is applied, itself to depend on $t$.
In order to formulate the main result in this part, Theorem \ref{thm-intro-main-ft} below, we need to replace the supremum by the norm of $\Lambda^\beta = \Lambda^\beta_{2,2}(\R_+) = \{ f : \R_+ \to \C : \: f \circ \exp \in W^\beta_2(\R)\}$, where $W^\beta_2(\R)$ is the usual Sobolev space.
The reason is that the method of our proofs uses Banach space geometric arguments, and with this respect, the Hilbert space $\Lambda^\beta$ is much nicer than $L^\infty(\R_+)$ and its norm $\sup_{t > 0}|f_t|$.
Yet $\Lambda^\beta$ embeds into $C_0(\R_+) \subseteq L^\infty(\R_+)$ for $\beta > \frac12$.
Note that Theorem \ref{thm-intro-main-ft} uses the space $W^c_1(\R)$ in contrast to the space $W^c_2(\R)$ in Theorem \ref{thm-intro-thm11-DK}.

\begin{thm}[see Theorem \ref{thm-main-ft}]
\label{thm-intro-main-ft}
Let $Y$ be a UMD lattice, $1 < p < \infty$ and $(\Omega,\mu)$ a $\sigma$-finite measure space.
Let $\beta \geq 0$.
Let $A$ be a $0$-sectorial operator on $L^p(Y)$.
Assume that $A$ has a $\Hor^\alpha_2$ calculus on $L^p(Y)$ for some $\alpha > \frac12$.
Let $m$ be a spectral multiplier with $m(0) = 0$ such that $\sum_{n \in \Z} \|m(2^n \cdot) \dyad_0 \|_{W^c_1(\R)} < \infty$, with
\[ c > \alpha + \max\left(\frac12,\frac{1}{\type L^p(Y)} - \frac{1}{\cotype L^p(Y)}\right) + 1 +
 \beta. \]
Then for $f_t \in L^p(Y(\Lambda^\beta_{2,2}(\R_+)))$, we have
\begin{align}
\label{equ-1-cor-main-ft-full-support}
\MoveEqLeft
\| t \mapsto m(tA)f_t\|_{L^p(Y(C_0(\R_+)))} \leq C' \| t \mapsto m(tA)f_t\|_{L^p(Y(\Lambda^\beta_{2,2}(\R_+)))} \\
& \leq C  \sum_{n \in \Z} \|m(2^n \cdot) \dyad_0\|_{W^{c}_1(\R)} \| t \mapsto f_t\|_{L^p(Y(\Lambda^\beta_{2,2}(\R_+)))}. \nonumber
\end{align}
\end{thm}

This seems to the best of our knowledge to be the first maximal estimate when the element $f = f_t$ is itself allowed to depend on $t$, even for the case $A = - \Delta$ the euclidean Laplacian.

In the main Theorems \ref{thm-intro-q-variation} resp. \ref{thm-intro-main-ft}, it is assumed that $m(0) = 0$.
However, we also have variants of these results in case $m(0) \neq 0$.
For the latter, we additionally need the assumption that the $q$-variational resp. maximal estimate holds for $e^{-tA}f$ on $L^p(\Omega,Y)$ .
This is guaranteed e.g. if $A = A_0 \ot \Id_Y$ where $e^{-tA_0}$ is lattice positive and contractive.

We deduce from Theorem \ref{thm-intro-main-ft} certain maximal estimates for the Schr\"odinger and wave solution propagator, in case that the semigroup generator has the form $A = A_0 \ot \Id_Y$ and under the hypothesis of $\Hor^\alpha_2$ calculus of $A$ and lattice positivity as well as contractivity of $\exp(-tA_0)$ on $L^p(\Omega)$ (see Corollary \ref{cor-wave}).
Then we consider the following abstract version of a problem by Carleson \cite[Theorem p. 24]{Car}.
Consider the abstract initial value problem 
\begin{equation}
\label{equ-Carleson-equation}
\begin{cases}
- i \frac{\partial}{\partial t} u(t,x,\omega) & = A_xu(t,x,\omega) \\
u(0,x,\omega) & =  f(x,\omega), \end{cases}
\end{equation}
where the solution $u$ depends on time $t$, space variable $x \in \Omega$ in which $A = A_x$ acts, and a mute variable $\omega \in \Omega'$.
If $A = - \Delta$ is the euclidean Laplacian on $\Omega = \R^d$, then \eqref{equ-Carleson-equation} is a Schr\"odinger equation and $\exp(itA)f$ is at least formally the solution of \eqref{equ-Carleson-equation}.
If $A = \sqrt{-\Delta}$, then $\exp(itA)f$ is the solution of a wave equation.
The solution for general initial value $f$ belonging, say, to $L^2(\R^d)$ is in general not continuous in $t \in [0,\infty)$ for fixed $(x,\omega) \in \Omega \times \Omega'$.
For the Schr\"odinger equation and without $Y$ component, Carleson asked for which differentiability parameter $\delta > 0$, $f$ belonging to $H^\delta(\R^d)$ ($L^2$ Sobolev space) implies this continuity.
For $d = 1$, he found $\delta > \frac14$ sufficient and $\delta \geq \frac18$ necessary.
Since then, Carleson's problem has been considered by many other authors, see \cite{BeGo,Bou91,Bou92,Bou13,Bou16,DaKe,DeGu,DGL,DGLZ,KPV,Lee,LuRo,MVV1,MVV2,RV1,RV2,Rog,Sjo,Sj1,Sj2,Tao,TV,TVV,Veg}
and references therein.
In this paper, we provide an abstract approach and consider equation \eqref{equ-Carleson-equation} for lattice positive and contractive semigroups having a H\"ormander calculus on $L^p(\Omega,Y)$.
Based on Theorem \ref{thm-intro-main-ft}, we show

\begin{thm}
\label{thm-intro-continuous-paths}[see Theorem \ref{thm-continuous-paths}]
Let $Y$ be a UMD lattice, $1 < p < \infty$ and $(\Omega,\mu)$ a $\sigma$-finite measure space.
Let $A = A_0 \ot \Id_Y$ have a $\Hor^\alpha_2$ calculus on $L^p(Y)$ for some $\alpha > \frac12$, and $A_0$ generate a lattice positive and contractive semigroup on $L^p(\Omega)$.
Pick
\[ \delta > \alpha + \max \left( \frac12, \frac{1}{\type L^p(Y)} - \frac{1}{\cotype L^p(Y)} \right) + \frac{3}{2} . \]
Then for any $f \in D(A^\delta)$, for a.e. $(x,\omega) \in \Omega \times \Omega'$, the function
\[ \begin{cases} (0,\infty) & \longmapsto \C \\ t & \longmapsto \exp(itA)f(x,\omega) \end{cases} \]
is continuous.
Similarly, for a.e. $x \in\Omega$, the function
\[ \begin{cases} (0,\infty) & \longmapsto Y \\ t & \longmapsto \exp(itA)f(x,\cdot) \end{cases} \]
is continuous.
\end{thm}

See also Corollary \ref{cor-continuity-on-the-real-line} for a version where continuity on the real line (in particular in $t = 0$) in place of on $(0,\infty)$ is obtained.
Note that the domain $D(A^\delta)$ reduces to $H^{2\delta}(\R^d)$ in Carleson's original problem $A = -\Delta$ on $L^2(\R^d)$.
The method of proof is to exploit an a priori maximal estimate (see Corollary \ref{cor-wave})
\[ \norm{\sup_{t \in [a,b]} | \exp(itA)f| \: }_{L^p(Y)} \lesssim \norm{(1+A)^\delta f}_{L^p(Y)} \]
which in turn is a consequence of Theorem \ref{thm-intro-main-ft}.
Note that our method is very general in that it allows an abstract operator $A$.
The drawback is that there is no reason to hope that the parameter $\delta$ is optimal.
Thus our approach does not reprove the recent solutions to Carleson's problem in the way he stated it originally.
However, we obtain several corollaries on variants of Carleson's problem, one of which improves the recent result \cite[Theorem 1.2]{Zhang} for dimensions $d > 8$.

\begin{cor}[see Corollary \ref{cor-Zhang}]
Let $0 < a < 2$, $d \in \N$ and $s > 2a$.
Let $A = (-\Delta)^{a/2}$ and $f \in H^s(\R^d)$.
Then for a.e. $x \in \R^d$, $\R \to \C,\: t \mapsto \exp(itA)f(x)$ is a continuous function.
In particular, $\exp(itA)f(x) \to f(x)$ as $t \to 0$ for a.e. $x \in \R^d$.
\end{cor}

See also the end of Section \ref{subsec-Carleson} for other applications on Carleson's continuity problem, with $A$ being a Fourier multiplier or an operator acting on an open domain $\Omega \subseteq \R^d$.

We end this introduction with an overview of the article.
In Section \ref{sec-preliminaries} we introduce the necessary background on Banach space geometry such as $R$-boundedness, UMD lattices, as well as type and cotype.
Moreover, we introduce the (restricted) necessary notions of H\"ormander functional calculus and the class $\Lambda^\beta_{2,2}(\R_+)$ that appears in the results above.
Then in Section \ref{sec-q-variation}, we state and prove our main result on finite $q$-variation of $t \mapsto m(tA)f(x,\omega)$.
Our method also allows under the same assumptions to show square function estimates of a family $m_1,\ldots,m_n$ of spectral multipliers satisfying uniformly the needed estimates (see Remark \ref{rem-square-function}) and to deduce jump estimates as an easy consequence (Corollary \ref{cor-jump}).
In the subsequent Section \ref{sec-q-variation-spherical}, we apply the $q$-variation result to the particular spectral multiplier $A_t$ of $-\Delta$ given by spherical mean over the sphere $t S^{d-1} \subseteq \R^d$.
We obtain Theorem \ref{thm-intro-variation-spherical} as a new reult and recover also partly the very recent result \eqref{equ-intro-variation-ball} from \cite{HHL}.
Section \ref{sec-wave} contains the main result Theorem \ref{thm-intro-main-ft} on the maximal estimate of $m(tA)f_t$ with the element $f_t$ depending itself on $t > 0$, as well as its proof.
In the last but one Section \ref{subsec-Carleson}, we apply Theorem \ref{thm-intro-main-ft} first to state and prove a maximal estimate of $t \mapsto \exp(itA)f$ and deduce Theorem \ref{thm-intro-continuous-paths} from it.
Then we show some applications to variants of Carleson's problem of pointwise convergence \cite{Car}
\[ \exp(itA)f(x,\omega) \to f(x,\omega) \text{ a. e. }(x,\omega) \quad (t \to 0+)\]
for $A$ a Fourier multiplier operator or a differential operator on an open domain $\Omega \subseteq \R^d$ with boundary conditions.
We conclude the article with the short Section \ref{sec-examples} on examples of operators $A = A_0 \otimes \Id_Y$ for which a H\"ormander calculus on $L^p(\Omega,Y)$ is known today, and some open questions linked with the subject of the present article.

\section{Preliminaries}
\label{sec-preliminaries}

In this section, we recall the notions on Banach space geometry and functional calculus that we need in this paper.
For the H\"ormander functional calculus, we only need few facts that we will use as an abstract blackbox in the remainder of the article.

\subsection{$R$-boundedness}
\label{subsec-R-boundedness}

\begin{defi}
Let $X,Y$ be Banach spaces.
We recall that a family $\tau \subseteq B(X,Y)$ is called $R$-bounded, if for a sequence $(\epsi_k)_k$ of independent Rademacher random variables, taking the value $1$ and $-1$ with equal probability $\frac12$, a constant $C > 0$, any $n \in \N$, any $x_1,\ldots,x_n \in X$ and any $T_1, \ldots , T_n \in \tau$, we have
\[ \E \left\| \sum_{k = 1}^n \epsi_k T_k x_k \right\|_Y \leq C \E \left\| \sum_{k = 1}^n \epsi_k x_k \right\|_X .\]
In this case, the infimum over all admissible $C$ is denoted by the $R$-bound $R(\tau)$.
\end{defi}

\begin{remark}
\label{rem-R-bdd-Hilbert}
Clearly, $R(\{T\}) = \|T\|_{B(X,Y)}$ if $\tau = \{ T \}$ is a singleton.
In general, we have $R(\tau) \geq \sup_{T \in \tau} \|T\|_{B(X,Y)}$ above.
If $X$ and $Y$ are (isomorphic to) Hilbert spaces, then a family $\tau \subset B(X,Y)$ is $R$-bounded if and only if $\tau$ is bounded, the latter meaning that $\sup_{T \in \tau} \|T\|_{B(X,Y)} < \infty$.
\end{remark}

\begin{defi}
\label{def-property-alpha-type-cotype}
Let $X$ be a Banach space and $(\epsi_n)_n$ be a sequence of independent Rademacher variables.
\begin{enumerate}
\item 
We say that $X$ has Pisier's property $(\alpha)$ if there are constants $c_1,c_2 > 0$ such that for any array $(x_{n,k})_{n,k = 1}^N$ in $X$,$(\epsi'_k)_k$ a second sequence of independent Rademacher variables independent of $(\epsi_n)_n$, and $(\epsi''_{n,k})_{n,k}$ a doubly indexed sequence of independent Rademacher variables, the following equivalence holds:
\[ c_1 \E \E' \left\| \sum_{k,n = 1}^N \epsi_n \epsi'_k x_{n,k} \right\|_{X} \leq \E'' \left\| \sum_{k,n = 1}^N \epsi''_{n,k} x_{n,k} \right\|_X \leq c_2 \E \E' \left\| \sum_{k,n = 1}^N \epsi_n \epsi'_k x_{n,k} \right\|_{X} . \] 
\item Let $p \in [1,2]$ and $q \in [2,\infty]$.
We say that $X$ has type $p$ if for some constant $c > 0$ and any sequence $(x_n)_{n = 1}^N$ in $X$, we have
\[ \E \left\| \sum_{n = 1}^N \epsi_n x_n \right\|_X \leq c \left( \sum_{n = 1}^N \|x_n\|^p \right)^{\frac1p} . \]
In this case, we write $\type(X) = p$ (not uniquely determined value).
We say that $X$ has cotype $q$ if for some constant $c > 0$ and any sequence $(x_n)_{n = 1}^N$ in $X$, we have
\[\left( \sum_{n = 1}^N \|x_n\|^q \right)^{\frac1q} \leq c  \E \left\| \sum_{n = 1}^N \epsi_n x_n \right\|_X  . \]
In this case, we write $\cotype(X) = q$ (not uniquely determined value).
\end{enumerate}
\end{defi}

\subsection{UMD lattices}
\label{subsec-UMD-lattices}

In this article, UMD lattices, i.e. Banach lattices which enjoy the UMD property, play a prevalent role.
For a general treatment of Banach lattices and their geometric properties, we refer the reader to \cite[Chapter 1]{LTz}.
We recall now definitions and some useful properties.
A Banach space $Y$ is called UMD space if the Hilbert transform 
\[ H : L^p(\R) \to L^p(\R),\: Hf(x) = \lim_{\epsilon \to 0} \int_{|x-y| \geq \epsilon} \frac{1}{x-y} f(y) \,dy \] 
extends to a bounded operator on $L^p(\R,Y),$ for some (equivalently for all) $1 < p < \infty$ \cite[Theorem 5.1]{HvNVW}.
The importance of the UMD property in harmonic analysis was recognized for the first time by Burkholder \cite{Burk1981,Burk1983}, see also his survey \cite{Burk2001}.
He settled a geometric characterization via a convex functional \cite{Burk1981} and together with Bourgain \cite{Bourgain1983}, they showed that the UMD property can be expressed by boundedness of $Y$-valued martingale sequences.
A UMD space is super-reflexive \cite{Al79}, and hence (almost by definition) B-convex.
Recall that a Banach space $X$ is called B-convex iff for some $\epsilon > 0$ and some natural number $n$, it holds true that whenever $x_1,\ldots,x_n$ are elements in the closed unit ball of $X$, then there is a choice of signs $\alpha_1,\ldots,\alpha_n \in \{-1,1\}$ such that $\| \sum_{i=1}^n \alpha_i x_i \| \leq (1 - \epsilon) n$.
As a survey for UMD lattices and their properties in connection with results in harmonic analysis, we refer the reader to \cite{RdF}.

A K\"othe function space $Y$ is a Banach lattice consisting of equivalence classes of locally integrable functions on some $\sigma$-finite measure space $(\Omega',\mu')$ with the additional properties
\begin{enumerate}
\item If $f :\: \Omega' \to \C$ is measurable and $g \in Y$ is such that $|f(\omega')| \leq |g(\omega')|$ for almost every $\omega' \in \Omega'$, then $f \in Y$ and $\|f\|_Y \leq \|g\|_Y$.
\item The indicator function $1_A$ is in $Y$ whenever $\mu'(A) < \infty$.
\item Moreover, we will assume that $Y$ has the $\sigma$-Fatou property:
if a sequence $(f_k)_k$ of non-negative functions in $Y$ satisfies $f_k(\omega') \nearrow f(\omega')$ for almost every $\omega' \in \Omega'$ and $\sup_k \|f_k\|_Y < \infty$, then $f \in Y$ and $\|f\|_Y = \lim_k \|f_k\|_Y$.
\end{enumerate}
Note that for example, any $L^p(\Omega')$ space with $1 \leq p \leq \infty$ is such a K\"othe function space.

\begin{lemma}
\label{lem-UMD-lattice-Fatou}
Let $Y$ be a UMD lattice.
Then it has the $\sigma$-Levi property: any increasing and norm-bounded sequence $(x_n)_n$ in $Y$ has a supremum in $Y$.
It also has the Fatou-property and hence the $\sigma$-Fatou property.
Note that if $1 < p < \infty$ and $(\Omega,\mu)$ is a $\sigma$-finite measure space, then $L^p(\Omega,Y)$ is again a UMD lattice, so has the above $\sigma$-Levi and $\sigma$-Fatou properties.
\end{lemma}

\begin{proof}
Note that a UMD lattice is reflexive.
Then we refer to \cite[Proposition B.1.8]{Lin}.
\end{proof}

\begin{ass}
In the rest of the paper, $Y = Y(\Omega')$ will always be a UMD space which is also a K\"othe function space, unless otherwise stated.
\end{ass}

\begin{defi}
\label{def-Lambda}
We define
\[ \Lambda^\beta := \Lambda^\beta_{2,2} := \Lambda^\beta_{2,2}(\R_+) := \{ f : \R_+ \to \C : \: f \circ \exp \text{ belongs to }W^\beta_2(\R) \} , \]
where $W^\beta_2(\R)$ denotes the usual Sobolev space defined e.g. via the Fourier transform.
We equip the space with the obvious norm $\|f\|_{\Lambda^\beta_{2,2}} := \|f \circ \exp\|_{W^\beta_2(\R)}$.
The space $\Lambda^\beta_{2,2}(\R_+)$ is a Hilbert space and imbeds into the (non-UMD) lattice $C_0(\R_+)$ for $\beta > \frac12$.
Indeed, this follows from the Sobolev embedding $W^\beta_2(\R) \hookrightarrow C_0(\R)$ for $\beta > \frac12$.
\end{defi}

Let $E$ be any Banach space.
We can consider the vector valued lattice $Y(E) = \{ F : \Omega' \to E :\: F \text{ is strongly measurable and }\omega' \mapsto \|F(\omega')\|_E \in Y\}$ with norm $\|F\|_{Y(E)} = \bigl\| \|F(\cdot)\|_E \bigr\|$.
From \cite[Corollary p.~214]{RdF}, we know that if $Y$ is UMD and $E$ is UMD, then also $Y(E)$ is UMD.
Moreover, we shall consider specifically in this article spaces $L^p(\Omega,Y(E))$, with e.g. $E = \Lambda^\beta$ as above.
For the natural identity $L^p(\Omega,Y)(E) = L^p(\Omega,Y(E))$ guaranteed e.g. by reflexivity of $Y$, we refer to \cite[Sections B.2.1, B.2.2, Theorem B.2.7]{Lin}.

\begin{remark}
\label{rem-Lambda-dilation-invariant}
The $\Lambda^\beta$ norm is dilation and inversion invariant, that is, for any $f \in \Lambda^\beta$ and $t > 0$, $\|f(t\cdot)\|_{\Lambda^\beta}  = \|f\|_{\Lambda^\beta}$ and $\left\| f\left(  \frac{1}{(\cdot)} \right) \right\|_{\Lambda^\beta} = \|f\|_{\Lambda^\beta}$.
Suppose $f : \R_+ \to \C$ is measurable and has compact support in $\R_+$.
Then $f$ belongs to $\Lambda^\beta$ iff $f$ belongs to $W^\beta_2(\R)$ and in this case, we have $\|f\|_{\Lambda^\beta} \cong \|f\|_{W^\beta_2(\R)}$, where the equivalence constants depend on the compact support.
\end{remark}

\begin{proof}
See \cite[Remark 2.7]{DK}.
\end{proof}

\begin{lemma}
\label{lem-UMD-lattice-Rademacher}
Let $Y = Y(\Omega')$ be a UMD lattice and  $(\epsilon_k)_k$ an i.i.d. Rademacher sequence.
Then we have the norm equivalence
\begin{equation}
\label{equ-Rademacher-square}
\E \biggl\|\sum_{ k = 1}^n \epsilon_k y_k \biggr\|_Y \cong \biggl\|\Bigl( \sum_{k = 1}^n |y_k|^2 \Bigr)^{\frac12}\biggr\|_Y
\end{equation}
uniformly in $n \in \N$.
In particular, this also applies to $L^p(\Omega,Y),\: 1 < p < \infty$.
\end{lemma}

\begin{proof}
As $Y$ is a UMD lattice, it is B-convex.
The result thus follows from \cite{Ma74}.
For the last sentence, we only need to recall that $L^p(\Omega,Y)$ will also be a B-convex Banach lattice.
\end{proof}

In the following, we will make use tacitly of the following Lemma \ref{lem-UMD-lattice-Hilbert-extension}.

\begin{lemma}
\label{lem-UMD-lattice-Hilbert-extension}
\begin{enumerate}
\item 
Let $T : Y \to Z$ be a bounded (linear) operator, where $Y(\Omega')$ and $Z(\Omega'')$ are B-convex Banach lattices.
Then its tensor extension $T \otimes \Id_{\ell^2},$ initially defined on $Y(\Omega') \otimes \ell^2 \subset Y(\Omega',\ell^2)$ is again bounded $Y(\Omega',\ell^2) \to Z(\Omega'',\ell^2).$
Also if $H$ is any Hilbert space isometric to $\ell^2$, then $T \otimes \Id_{H}$ extends to a bounded operator $Y(\Omega',H) \to Z(\Omega'',H)$.
In particular, if $Y(\Omega')$ is a UMD lattice, then $Y(\Omega',\ell^2)$ is also a UMD lattice.
\item Let $Y(\Omega')$ be a B-convex Banach lattice and $H$ a Hilbert space.
Then $\type Y(H) = \type Y$ and $\cotype Y(H) = \cotype Y$, where $\type \in (1,2]$ and $\cotype \in [2,\infty)$.
\end{enumerate}
\end{lemma}

\begin{proof}
See \cite[Lemma 2.9]{DK}.
In 1., for the boundedness of $T \otimes \Id_{H}$, it suffices to consider an isometry $\Psi : \ell^2 \to H$ and to write $T \otimes \Id_{H} = (\Id_Z \otimes \Psi )(T \otimes \Id_{\ell^2})(\Id_Y \otimes \Psi^{-1})$ and to note that $\Id_Z \otimes \Psi$ and $\Id_Y \otimes \Psi^{-1}$ are again isometries.
\end{proof}

The following lemma will be used in combination with Proposition \ref{prop-Hormander-calculus-to-R-Hormander-calculus} to follow.

\begin{lemma}
\label{lem-property-alpha}
Let $Y$ be a UMD lattice and $p \in (1,\infty)$.
Then $L^p(\Omega,Y)$ has Pisier's property $(\alpha)$.
\end{lemma}

\begin{proof}
Note that $L^p(\Omega,Y)$ is a UMD lattice, so it has Pisier's property $(\alpha)$, see \cite[Proposition 7.5.4 and Theorem 7.5.20]{HvNVW2}.
\end{proof}

\subsection{Abstract H\"ormander functional calculus}
\label{subsec-abstract-Hormander}

We recall the necessary background on functional calculus that we will treat in this article.
Let $-A$ be a generator of an analytic semigroup $(T_z)_{z \in \Sigma_\delta}$ on some Banach space $X,$ that is, $\delta \in (0,\frac{\pi}{2}],$ $\Sigma_\delta = \{ z \in \C \backslash \{ 0 \} :\: | \arg z | < \delta \},$ the mapping $z \mapsto T_z$ from $\Sigma_\delta$ to $B(X)$ is analytic, $T_{z+w} = T_z T_w$ for any $z,w \in \Sigma_\delta,$ and $\lim_{z \in \Sigma_{\delta'},\:|z| \to 0 } T_zx = x$ for any $x \in X$ and any strict subsector $\Sigma_{\delta'}$ of $\Sigma_\delta$.
We assume that $(T_z)_{z \in \Sigma_\delta}$ is a bounded analytic semigroup, which means $\sup_{z \in \Sigma_{\delta'}} \|T_z\| < \infty$ for any $\delta' < \delta.$

It is well-known \cite[Theorem 4.6, p. 101]{EN} that this is equivalent to $A$ being $\omega$-sectorial for $\omega = \frac{\pi}{2} - \delta,$ that is,
\begin{enumerate}
\item $A$ is closed and densely defined on $X;$
\item The spectrum $\sigma(A)$ is contained in $\overline{\Sigma_\omega}$ (in $[0,\infty)$ if $\omega = 0$);
\item For any $\omega' > \omega,$ we have $\sup_{\lambda \in \C \backslash \overline{\Sigma_{\omega'}}} \| \lambda (\lambda - A)^{-1} \| < \infty.$
\end{enumerate}
We say that $A$ is strongly $\omega$-sectorial if it is $\omega$-sectorial and has moreover dense range.
If $A$ is $\omega$-sectorial and does not have dense range, but $X$ is reflexive, which will always be the case in this article, then we may take the injective part $A_0$ of $A$ on $\overline{R(A)} \subseteq X$ \cite[Proposition 15.2]{KW04}, which then does have dense range and is strongly $\omega$-sectorial.
Here, $R(A)$ stands for the range of $A.$
Then $-A$ generates an analytic semigroup on $X$ if and only if so does $-A_0$ on $\overline{R(A)}.$
For $\theta \in (0,\pi),$ let 
\[ \HI(\Sigma_\theta) = \{ f : \Sigma_\theta \to \C :\: f \text{ analytic and bounded} \} \] equipped with the uniform norm $\|f\|_{\infty,\theta} = \sup_{z \in \Sigma_\theta} |f(z)|.$
Let further 
\[ \HI_0(\Sigma_\theta) = \bigl\{ f \in \HI(\Sigma_\theta):\: \exists \: C ,\epsilon > 0 \text{ such that } |f(z)| \leq C \min(|z|^\epsilon,|z|^{-\epsilon}) \bigr\}.\]
For an $\omega$-sectorial operator $A$ and $\theta \in (\omega,\pi),$ one can define a functional calculus $\HI_0(\Sigma_\theta) \to B(X),\: f \mapsto f(A)$ extending the ad hoc rational calculus, by using a Cauchy integral formula.
Moreover, if there exists a constant $C < \infty$ such that $\|f(A)\| \leq C \| f \|_{\infty,\theta},$ then $A$ is said to have a bounded $\HI(\Sigma_\theta)$ calculus and if $A$ has a dense range, the above functional calculus can be extended to a bounded Banach algebra homomorphism $\HI(\Sigma_\theta) \to B(X).$
If $A$ has a bounded $\HI(\Sigma_\theta)$ calculus, and does not have dense range, but $X$ is reflexive, then for $f \in \HI(\Sigma_\theta)$ such that $f(0)$ is well-defined, we can define
\[ f(A) = \begin{bmatrix} f(A_0) & 0 \\ 0 & f(0) P_{N(A)} \end{bmatrix} : \: \overline{R(A)} \oplus N(A) \to \overline{R(A)} \oplus N(A) ,\]
where $P_{N(A)}$ denotes the projection onto the null-space of $A$ along the decomposition $X = \overline{R(A)} \oplus N(A)$.
This calculus also has the property $f_z(A) = T_z$ for $f_z(\lambda) = \exp(-z \lambda),\: z \in \Sigma_{\frac{\pi}{2} - \theta}.$
For further information on the $\HI$ calculus, we refer e.g. to \cite{KW04}. We now turn to H\"ormander function classes and their calculi.
\begin{defi}
\label{defi-Hoermander-class}
Let $\alpha > \frac12.$
We define the H\"ormander class by 
\[\Hor^\alpha_2 = \bigl\{ f : [0,\infty) \to \C \text{ is bounded and continuous on }(0,\infty), \:  \underbrace{|f(0)| + \sup_{R > 0} \| \phi f(R \,\cdot) \|_{W^\alpha_2(\R)}}_{=:\|f\|_{\Hor^\alpha_2}}< \infty \bigr\}.\]
Here $\phi$ is any $C^\infty_c(0,\infty)$ function different from the constant 0 function (different choices of functions $\phi$ resulting in equivalent norms) and $W^\alpha_2(\R)$ is the classical Sobolev space.
\end{defi}

The term $|f(0)|$ is not needed in the functional calculus applications of $\Hor^\alpha_2$ if $A$ is in addition injective.
We can base a H\"ormander functional calculus on the $\HI$ calculus by the following procedure.
Note that for any $\alpha > \frac12$ and any $\theta \in (0,\pi)$, $\HI(\Sigma_\theta)$ injects continuously into $\Hor^\alpha_2$.

\begin{defi}
\label{def-Hormander-calculus}
We say that a $0$-sectorial operator $A$ has a bounded $\Hor^\alpha_2$ calculus if for some $\theta \in (0,\pi)$ and any $f \in \HI(\Sigma_\theta),$
$\|f(A)\| \leq C \|f\|_{\Hor^\alpha_2} \left( \leq C' \left(\|f\|_{\infty,\theta} + |f(0)|\right) \right).$
In this case, the $\HI(\Sigma_\theta)$ calculus can be extended to a bounded Banach algebra homomorphism $\Hor^\alpha_2 \to B(X)$ \cite{KrW3}.
We say that $A$ has an $R$-bounded $\Hor^\alpha_2$ calculus, if it has a bounded $\Hor^\alpha_2$ calculus and
$\left\{ m(A) : \: \|m\|_{\Hor^\alpha_2} \leq 1 \right\}$ is $R$-bounded.
\end{defi}

The H\"ormander norm is dilation invariant, i.e. $\|f(t \cdot)\|_{\Hor^\alpha_2} = \|f\|_{\Hor^\alpha_2}$ for any $t > 0$.
Therefore, the following family of (discrete) dilates of a $C^\infty_c(\R_+)$ function will play an important role.

\begin{defi}
\label{def-dyad}
Let $\dyad_0 \in C^\infty_c(\R_+)$ such that $\supp (\dyad_0) \subseteq [\frac12,2]$.
We define for $n \in \Z$ the dilates $\dyad_n(t) = \dyad_0(2^{-n}t)$ so that $\supp (\dyad_n) \subseteq [\frac12 \cdot 2^n , 2 \cdot 2^n]$.
Assume that $\sum_{n \in \Z} \dyad_n(t) = 1$ for any $t > 0$.
Then we call $(\dyad_n)_{n \in \Z}$ a dyadic partition of $\R_+ = (0,\infty)$.
For the existence of such a dyadic partition, we refer to \cite[6.1.7 Lemma]{BeL}.
\end{defi}

In the course of the main theorems \ref{thm-q-variation} and \ref{thm-main-ft}, we need to decompose general spectral multipliers by means of special spectral multiplier pieces involving the above dyadic partition.
To reassemble the pieces together, mere boundedness of the pieces is not sufficient, and we will need the following self-improvement of a H\"ormander functional calculus.

\begin{prop}
\label{prop-Hormander-calculus-to-R-Hormander-calculus}
Let $A$ be a $0$-sectorial operator on a Banach space $X$ with property $(\alpha)$.
If $A$ has a bounded $\Hor^\alpha_2$ calculus, then it has an $R$-bounded $\Hor^\gamma_2$ calculus
for any parameter $\gamma > \alpha + \frac{1}{\type X} - \frac{1}{\cotype X}$ such that $\gamma \geq \alpha + \frac12$.
If $X$ is a Hilbert space and $A$ has a bounded $\Hor^\alpha_2$ calclus, then $A$ has an $R$-bounded $\Hor^\alpha_2$ calculus.
\end{prop}

\begin{proof}
This follows from \cite[Lemma 3.9 (3), Theorem 6.1 (2)]{KrW3}, noting that the $\Hor^\beta_r$ class there is larger than our $\Hor^\gamma_2$ class for $\gamma = \beta$.
The last sentence follows from Remark \ref{rem-R-bdd-Hilbert}.
\end{proof}

The following lemmata concerning decomposition/expansion of spectral multipliers will be used in the proof of Theorem \ref{thm-q-variation}.
Here, Lemma \ref{lem-Hormander-calculus-Paley-Littlewood} is sometimes called Paley-Littlewood equivalence.

\begin{lemma}
\label{lem-Hormander-convergence-lemma}
Let $A$ be a $0$-sectorial operator with $\Hor^\alpha_2$ calculus.
Let $(\dyad_n)_{n \in \Z}$ be a dyadic partition of $\R_+$.
Then for any $x \in \overline{R(A)}$ (e.g. $x = m(A)y$ for some $y \in X$ and $m \in \Hor^\alpha_2$ with $m(0) = 0$), we have $x = \sum_{n \in \Z} \dyad_n(A) x$ (convergence in $X$).
\end{lemma}

\begin{proof}
See \cite[Corollary 4.20]{KrPhD}.
\end{proof}

In the setting of the above Lemma \ref{lem-Hormander-convergence-lemma}, we obtain that $D_A := \{ \phi(A)x :\: x \in X,\: \phi \in C^\infty_c(\R_+) \}$ is a dense subspace of $\overline{R(A)}$.
In \cite{KrW3}, $D_A$ is called the calculus core of $A$.

\begin{lemma}
\label{lem-representation-formula-wave-operators}
Let $A$ be a $0$-sectorial operator having a $\Hor^\alpha_2$ calculus.
Let $m \in W^\alpha_2(\R)$ with compact support in $\R_+$.
Then for any $x$ belonging to the calculus core $D_A$, we have
\[ m(A)x = \frac{1}{2\pi} \int_\R \hat{m}(s) \exp(isA) x ds ,\]
where the integral is a Bochner integral in $X$.
\end{lemma}

\begin{proof}
This follows from \cite[Proof of Lemma 4.6 (3)]{KrW3}.
\end{proof}

\begin{lemma}
\label{lem-Hormander-calculus-Paley-Littlewood}
Let $A$ be a $0$-sectorial operator with $\Hor^\alpha_2$ calculus for some $\alpha > \frac12$.
Let $(\dyad_n)_{n \in \Z}$ be a dyadic partition of $\R_+$.
Then we have the following so-called Paley-Littlewood decomposition for $x \in \overline{R(A)}$:
\[ \|x\|_X \cong \E \left\| \sum_{n \in \Z} \epsi_n \dyad_n(A) x \right\|_X ,\]
where the series $\sum_{n \in \Z} \dyad_n(A)x$ converges unconditionally in $X$.
\end{lemma}

\begin{proof}
See \cite[Theorem 4.1]{KrW2}, together with the fact that the restriction of $A$ to $\overline{R(A)}$ is a strongly $0$-sectorial operator having a $\Hor^\alpha_2$ calculus, hence a $\mathcal{M}^\beta$ calculus \cite[Proposition 4.9]{KrPhD} needed in this reference.
\end{proof}

\section{$q$-variational inequalities}
\label{sec-q-variation}

In this section, we consider $q$-variational inequalities for H\"ormander spectral multipliers.
The proofs are based on \cite[Section 3]{DK}.
We refer to the recent papers \cite{LMX}, \cite{HoMa} for $q$-variational inequalities for the particular spectral multiplier $\exp(-tA)$ for contractive regular analytic semigroups.
\begin{defi}
\label{def-q-variation}
Let $q \in [2 ,\infty)$.
For a function $a : \R_+ \to \C$, we define the $q$-variation
\[ \|a\|_{V^q} = \sup \left\{ \left( |a_{t_0}|^q + \sum_{k = 1}^\infty |a_{t_k} - a_{t_{k-1}}|^q \right)^{\frac1q} \right\} , \]
where the supremum runs over all increasing sequences $(t_k)_{k \in \N}$ in $\R_+$.
Then the space $V^q$ consisting of all functions with finite $q$-variation is a Banach space \cite[Section 1]{LMX}.
\end{defi}

The following simple observation is at the heart of the proof of Theorem \ref{thm-q-variation} below.

\begin{lemma}
\label{lem-q-variation}
Let $q \in [2,\infty)$ and $\beta > \frac12$.
Then $\Lambda^\beta_{2,2}(\R_+) \hookrightarrow V^q$, where we take for the equivalence class of an element $a \in \Lambda^\beta_{2,2}(\R_+)$ the (unique) continuous representative.
\end{lemma}

\begin{proof}
It can be seen directly from Definition \ref{def-q-variation} that $V^q \hookrightarrow V^r$ if $q \leq r$ (same proof as the fact that $\ell^q \hookrightarrow \ell^r$ for $q \leq r$), so that it suffices to pick the smallest considered exponent $q = 2$.
According to \cite[Theorem 5]{BLCS} with $p = 2$ there (see also \cite{Pee}), we have $B^{\frac12}_{2,1}(\R) \hookrightarrow \dot{B}^{\frac12}_{2,1}(\R) \hookrightarrow BV_2(\R)$.
The latter space $BV_2(\R)$ is defined in \cite[p.~461]{BLCS} and it can be seen from that source that $a \in V^2$ iff $a_e = a \circ \exp \in BV_2(\R)$.
So we have for $a \in \Lambda^\beta_{2,2}(\R_+)$, using the above injection and the well-known Besov space fine index injection $B^\beta_{2,2}(\R) \hookrightarrow B^{\frac12}_{2,1}(\R)$ \cite[2.3.2 Proposition 2]{Tri} with $\beta > \frac12$,
\[ \|a\|_{V^2} \cong \|a_e\|_{BV_2(\R)} \lesssim \|a_e\|_{B^{\frac12}_{2,1}(\R)} \lesssim \|a_e\|_{B^{\beta}_{2,2}(\R)} \cong \|a_e\|_{W^\beta_{2}(\R)} = \|a\|_{\Lambda^\beta_{2,2}(\R_+)} . \]
\end{proof}

We have the following variant of \cite[Proposition 3.9]{DK}, which yields finite $q$-variation spectral multipliers, thus extending \cite[(1.3)]{LMX} resp. \cite[(1.3) arXiv version]{HoMa} from the semigroup on scalar resp. UMD lattice valued $L^p$ spaces to more general spectral multipliers in case that $A$ has a H\"ormander calculus.

\begin{thm}
\label{thm-q-variation}
Let $Y = Y(\Omega')$ be a UMD lattice, $1 < p < \infty$ and $(\Omega,\mu)$ a $\sigma$-finite measure space.
Let $A$ be a $0$-sectorial operator on $L^p(\Omega,Y)$ having a $\Hor^\alpha_2$ calculus.
Let
\[ c > \alpha + \max\left(\frac12,\frac{1}{\type L^p(Y)} - \frac{1}{\cotype L^p(Y)} \right) + 1. \]
Let $m$ be a spectral multiplier such that
\begin{enumerate}
\item $m \in W^c_2(\R)$ and $\supp (m) \in [\frac12,2]$ or, more generally,
\item $m(0) = 0$, $m(2^n \cdot)\dyad_0 \in W^c_2(\R)$ for all $n \in \Z$ and $\sum_{n \in \Z} \|m(2^n \cdot)\dyad_0\|_{W^c_2(\R)} < \infty$, where $(\dyad_n)_{n \in \Z}$ is a partition of unity of $\R_+$ or,
\item assume that $A$ is of the form $A = A_0 \otimes \Id_Y$, where $A_0$ is $0$-sectorial on $L^p(\Omega)$ and that the semigroup $\exp(-tA_0)$ generated by $-A_0$ is lattice positive and contractive on $L^p(\Omega)$ (more generally regular contractive on $L^p(\Omega)$).
Assume $m|_{[0,1]} \in C^1[0,1] , \: \|m' \cutoff_0 \|_{\Hor^{c-1}_2} < \infty , \: \sum_{n \geq 0} \|m(2^n \cdot) \dyad_0\|_{W^c_2(\R)} < \infty$,
for some $\cutoff_0 \in C^\infty(\R_+)$ with support in $(0,4]$ and equal to $1$ on $(0,2]$.
\end{enumerate}
Then for $2 < q < \infty$, any $f \in L^p(\Omega,Y)$ and a.e. $(x,\omega) \in \Omega \times \Omega'$, $t \mapsto m(tA)f(x,\omega)$ has finite $q$-variation, and
\[ \left\| \, \|t \mapsto m(tA)f\|_{V^q} \right\|_{L^p(Y)} \leq C_m \|f\|_{L^p(Y)} \]
with
\begin{enumerate}
\item $C_m \leq C \norm{m}_{W^c_2(\R)}$,
\item $C_m \leq C \sum_{n \in \Z} \norm{m(2^n \cdot) \dyad_0}_{W^c_2(\R)}$,
\item $C_m \leq C \left( |m(0)| + \|m' \cutoff_0\|_{\Hor^{c-1}_2} + \sum_{n\geq 0} \|m(2^n \cdot) \dyad_0 \|_{W^c_2(\R)} \right)$.
\end{enumerate}
\end{thm}

\begin{proof}
The proof goes along the same lines as that of \cite[Corollary 3.5, Proposition 3.9]{DK}, but using the embedding $\Lambda^\beta \hookrightarrow V^q$ for $\beta > \frac12$ from Lemma \ref{lem-q-variation} in place of the embedding $\Lambda^\beta \hookrightarrow C_0(\R_+)$.
We recall the main ingredients for each of the three cases above.

1. First, consider $m$ compactly supported in $[\frac12,2]$.
Then the estimate follows from Lemma \ref{lem-q-variation} together with \cite[Theorem 3.1]{DK}.

2. If $m$ is not compactly supported but $\sum_{n \in \Z} \|m(2^n \cdot)\dyad_0\|_{W^c_2(\R)} < \infty$, then decompose
\[\norm{t \mapsto m(tA)f}_{L^p(Y(V^q))} \leq \sum_{n \in \Z} \norm{t \mapsto \dyad_n(tA)m(tA)f }_{L^p(Y(V^q))}\]
and use part 1. on the dyadically supported pieces, noting that the $V^q$-norm is invariant under dilations $f \mapsto f(t\cdot)$ for $t > 0$.

3. Then for those spectral multipliers not converging to $0$ at $0$, we decompose $m(t) = m_1(t) + m(0) \exp(-t)$ where one checks by the assumptions on $m$ that $m_1$ does satisfy
\[\sum_{n \in \Z} \|m_1(2^n \cdot)\dyad_0\|_{W^c_2(\R)} < \infty. \]
Now decompose 
\[ \norm{t \mapsto m(tA)f}_{L^p(Y(V^q))} \leq \norm{t \mapsto m_1(tA)f}_{L^p(Y(V^q))} + |m(0)| \norm{t \mapsto \exp(-tA)f}_{L^p(Y(V^q))} . \]
Use the second case to estimate the part with $m_1$.
Finally, use \cite[(1.3)]{LMX} (case $Y = \C$) or \cite[(1.3) arXiv version]{HoMa} (case of general $Y$) to estimate the $q$-variation of the semigroup.
\end{proof}

\begin{remark}
\label{rem-square-function}
Under the assumptions of Theorem \ref{thm-q-variation} above, we can also formulate a square function estimate for a family of spectral multipliers $(m_k(tA))_k$.
We then obtain with $c$ as in Theorem \ref{thm-q-variation},
\begin{align*}
\MoveEqLeft
\norm{ \left( \sum_k \norm{t \mapsto m_k(tA)f_k}_{V^q}^2 \right)^{\frac12} }_{L^p(Y)} \lesssim \\
& \left( \sup_k |m_k(0)| + \sup_k \norm{m_k' \cutoff_0}_{\Hor^{c-1}_2} + \sum_{n\geq 0} \sup_k \|m_k(2^n \cdot) \dyad_0 \|_{W^c_2(\R)} \right) \norm{ \left( \sum_k |f_k|^2 \right)^{\frac12} }_{L^p(Y)} .
\end{align*}
The proof goes along the same lines as that of Theorem \ref{thm-q-variation}, using \cite[Corollary 3.3]{DK} and the embedding $\Lambda^\beta \hookrightarrow V^q$ from Lemma \ref{lem-q-variation}.
\end{remark}

According to \cite[Corollary 3.6]{DK}, under the assumptions of Theorem \ref{thm-q-variation}, $m(tA)f(x,\omega)$ converges pointwise a.e. $(x,\omega) \in \Omega \times \Omega'$ to $m(0)Pf(x,\omega)$ as $t \to 0+$ or $t \to \infty$, where $P : L^p(Y) \to L^p(Y)$ denotes the projection onto the null-space of $A$.
The $q$-variation norm can be used to quantify this convergence.
Namely, for $a : \R_+ \to \C$ a function and $\lambda > 0$, we define the jump quantity $N(a,\lambda)$ to be the supremum of all integers $N \geq 0$ such that there is an increasing sequence
\[ 0 < s_1 < t_1 \leq s_2 < t_2 \leq \ldots \leq s_N < t_N \]
so that $|a(t_k) - a(s_k)| > \lambda$ for each $k = 1 , \ldots , N$.
It is clear from Definition \ref{def-q-variation} that $\lambda^q N(a,\lambda) \leq \|a\|_{V^q}^q$.
We thus obtain the following result for jump estimates.

\begin{cor}
\label{cor-jump}
Let the assumptions of Theorem \ref{thm-q-variation} hold.
Then for $2 < q < \infty$ and any spectral multiplier $m$ as in that proposition, if the function under the norm of the left hand side below is measurable, then
\[ \left\| (x,\omega) \mapsto N(m(tA)f(x,\omega),\lambda)^{\frac1q} \right\|_{L^p(Y)} \lesssim \frac{\|f\|_{L^p(Y)}}{\lambda} . \]
\end{cor}

\section{$q$-variation for spherical and euclidean ball means in $\R^d$}
\label{sec-q-variation-spherical}

In this section, we consider for $t > 0$, $d \in \N$, $1 < p < \infty$ and $f \in S(\R^d)$ a Schwartz function,
\begin{align}
A_t f(x) & = \frac{1}{|S_d|} \int_{S_d}f(x - t y) d\sigma(y), \label{equ-def-spherical-mean} \\
M_t f(x) & = \frac{1}{|B_d|} \int_{B_d}f(x - t y) dy = \frac{1}{|B(x,t)|} \int_{B(x,t)} f(y) dy, \label{equ-def-ball-mean}
\end{align}
where $S_d$ (resp. $B_d$) denotes the unit sphere (resp. the unit euclidean ball) $\subseteq \R^d$, $|S_d|$ (resp. $|B_d|$) denotes the surface measure of $S_d$ (resp. the Lebesgue measure of $B_d$) and $d\sigma(y)$ (resp. $dy$) denotes surface measure (resp. euclidean measure).
The definitions \eqref{equ-def-spherical-mean} and \eqref{equ-def-ball-mean} extend literally to $f \in S(\R^d) \otimes Y$ when $Y$ is a UMD lattice.
They moreover extend by density to $f \in L^p(\R^d,Y)$ once a priori boundedness of $A_t$ and $M_t$ on $L^p(\R^d,Y)$ is clarified.
In this section, we shall show $q$-variation results of $A_t$ and $M_t$ on $L^p(\R^d,Y)$.
We recall the following results from the literature.

\begin{thm}
\begin{enumerate}
\item (See \cite[Theorem 1.4(i)]{JSW}) Let $2 < q < \infty$, $d \geq 2$, $\frac{d}{d-1} < p \leq 2d$ and $f \in S(\R^d)$.
There exists a constant $C_{p,q,d}$ independent of $f$ such that
\begin{equation}
\label{equ-JSW}
\|t \mapsto A_tf\|_{L^p(\R^d,V^q)} \leq C_{p,q,d} \|f\|_{L^p(\R^d)} .
\end{equation}
Moreover, if $p > 2d$ and $q > \frac{p}{d}$, then \eqref{equ-JSW} holds and conversely, if \eqref{equ-JSW} holds, then necessarily $q \geq \frac{p}{d}$.
\item (See \cite[Theorem 1.3]{BMSW})  Let $2 < q < \infty$, $d \geq 1$, $1 < p < \infty$ and $f \in S(\R^d)$.
There exists a constant $C_{p,q}$ independent of $f$ and the dimension $d$ such that
\begin{equation}
\label{equ-BMSW}
\|t \mapsto M_tf\|_{L^p(\R^d,V^q)} \leq C_{p,q} \|f\|_{L^p(\R^d)} .
\end{equation}
Moreover, the same result holds if the euclidean $\ell^2_d$ ball in the definition \eqref{equ-def-ball-mean} is replaced by an $\ell^r_d$-ball with $1 \leq r \leq \infty$.
\item (See \cite[Theorem 1.1]{HHL}) Let $2 < q < \infty$, $d \geq 1$, $1 < p < \infty$, $Y$ a UMD lattice and $f \in S(\R^d) \otimes Y$.
Then there exists a constant $C_{p,q,Y}$ independent of $f$ and the dimension $d$ such that
\begin{equation}
\label{equ-HHL}
\|t \mapsto M_tf\|_{L^p(\R^d,Y(V^q))} \leq C_{p,q,Y} \|f\|_{L^p(\R^d,Y)} .
\end{equation}
\end{enumerate}
\end{thm}

We shall prove the following extension of \eqref{equ-JSW} to the UMD lattice valued case.
In order \eqref{equ-JSW} to hold, we had a minimal dimension $d > d_1 = \max\left( \frac{p}{2}, \frac{p}{p-1}, \frac{p}{q} \right)$.
In Theorem \ref{thm-spherical-means}, we also need some minimal dimension, but this time, it also depends on the geometry of the UMD lattice.

\begin{thm}
\label{thm-spherical-means}
Let $2 < q < \infty$, $1 < p < \infty$, $Y$ be a UMD lattice and 
\begin{equation}
\label{equ-minimal-dimension}
d > d_0 := \frac{2 + \max\left(\frac12, \frac{1}{\type L^p(Y)} - \frac{1}{\cotype L^p(Y)}\right)}{\frac12 - \max\left( \left| \frac1p - \frac12 \right|, \left| \frac{1}{p_Y} - \frac12 \right|, \left| \frac{1}{q_Y} - \frac12 \right| \right)} ,
\end{equation}
where $p_Y \in (1,2]$ and $q_Y \in [2,\infty)$ stand for the convexity and concavity exponents of $Y$ (see \cite[Definition 1.d.3]{LTz}).
Let $f \in S(\R^d) \otimes Y$.
Then there exists a constant $C_{p,q,d,Y}$ independent of $f$ such that
\[ \|t \mapsto A_tf\|_{L^p(\R^d(Y(V^q)))} \leq C_{p,q,d,Y} \|f\|_{L^p(\R^d,Y)} . \]
\end{thm}

\begin{proof}
The idea of the proof is the following.
The operator $A_t$ commutes with translations in $\R^d$, so is a Fourier multiplier associated to some symbol $m : \R^d \to \C$.
Since $A_t$ is rotation invariant, also $m$ is rotation invariant and thus, $A_t$ is a spectral multiplier of the Laplacian $- \Delta$ on $\R^d$ (or its square root $\sqrt{-\Delta}$).
For $\sqrt{-\Delta}$, a H\"ormander functional calculus on $L^p(\R^d,Y)$ space has been established in \cite{Hy1,GiWe}, see also \cite{DKK}.
So Theorem \ref{thm-q-variation} is available.
It will then suffice to check the finiteness of the norms of the specific spectral multiplier function $m$ imposed in this Theorem \ref{thm-q-variation}.

In details, we have $A_tf(x) = m(t\sqrt{-\Delta})f(x)$ for
\begin{equation}
\label{equ-spectral-multiplier-spherical-mean}
m(\lambda) = \frac{2\pi}{\lambda^{\frac{d-2}{2}}} J_{\frac{d-2}{2}}(2 \pi \lambda),
\end{equation}
where $J_{\alpha}$ denotes the Bessel function of order $\alpha$ \cite[p.~577]{Graf}.
Indeed, with $f_t(x) := f(tx)$ for $f$ say, a Schwartz function,
\begin{align*}
\MoveEqLeft
A_tf(x) \overset{\eqref{equ-def-spherical-mean}}{=} \frac{1}{|S_d|} \int_{S_d}f(x - t y) d\sigma(y) = \frac{1}{|S_d|} \int_{S_d} f_t\left(\frac{1}{t}x - y\right) d \sigma(y) \\
& = \frac{1}{|S_d|} f_t \ast \sigma\left(\frac{1}{t}x\right) = \frac{1}{|S_d|} \mathcal{F}^{-1}[ (f_t)\hat{\phantom{i}} \hat{\sigma}]\left(\frac{1}{t} x\right) \\
& \overset{\cite[p.~577]{Graf}}{=} \mathcal{F}^{-1}[(f_t) \hat{\phantom{i}} m]\left(\frac{1}{t}x\right) \\
& = \mathcal{F}^{-1}[\hat{f} m(t|\cdot|)](x) = m(t\sqrt{-\Delta})f(x).
\end{align*}
We note the two following crucial formulae concerning the Bessel function:
\begin{align}
\sup_{\lambda \geq 0} & \sqrt{\lambda} \left|J_\alpha(\lambda) \right| < \infty \quad \left(\alpha > \frac12 \right) \label{equ-Bessel-1}, \\
\frac{d}{d\lambda} \left(\lambda^{-\alpha} J_\alpha(\lambda) \right) & = - \lambda^{-\alpha} J_{\alpha + 1}(\lambda) \quad \left(\alpha > \frac12 \right) \label{equ-Bessel-2},
\end{align}
where we refer for \eqref{equ-Bessel-1} to \cite[p.~238]{AAR} and for \eqref{equ-Bessel-2} to \cite[p.~573]{Graf}.
We next claim that for $c < \frac{d-1}{2}$, the spectral multiplier function satisfies the hypotheses of Theorem \ref{thm-q-variation}.
First we note that for $\cutoff_0$ a $C^\infty$ function with support in $(0,4]$ and equal to $1$ on $(0,2]$, $m(\lambda)\cutoff_0(\lambda)$ extends to a $C^\infty$ function on $\R$ \cite[p.~577]{Graf}.
Thus, $m \cutoff_0 \in C^1[0,1]$ and $m' \chi_0 \in \Hor^{c-1}_2$ for any $c$ since $m'$ has bounded derivatives on intervals of finite length, to any order.
To see this, note that the function $\lambda^{-\alpha}J_\alpha(\lambda)$ is $C^\infty$ on $(0,\infty)$ and continuous in $0$ being essentially the Fourier transform of surface measure, with $\alpha= \frac{d-2}{2}$.
Then so is also $m$.
Now \eqref{equ-Bessel-2} yields $\frac{d}{d\lambda}(\lambda^{-\alpha} J_\alpha(\lambda)) = - \lambda (\lambda^{-(\alpha+1)} J_{\alpha+1}(\lambda))$.
Thus, also $m'$ is continuous in $0$.
Now using \eqref{equ-Bessel-2} with $\alpha$ replaced by $\alpha + 1$ shows that also $m''$ is continuous in $0$.
Repeating this argument by incrementing $\alpha$ step by step shows that $m'$ has bounded derivatives on intervals of finite length, to any order.
In particular, $m(0)$ takes some finite value.
We now estimate $\sum_{n \geq 0} \|m(2^n\cdot) \dyad_0\|_{W^c_2(\R)}$ for $(\dyad_n)_n$ a partition of unity of $\R_+$.
We have according to \eqref{equ-Bessel-2} and \eqref{equ-Bessel-1} for $\lambda \geq \frac12$
\[ \left|\frac{d}{d\lambda}m(\lambda) \right| =  c \lambda^{-\frac{d-2}{2}} \left| J_{\frac{d}{2}}(2\pi \lambda) \right| \lesssim \lambda^{-\frac{d-1}{2}} . \]
Moreover,
\begin{align*}
\MoveEqLeft
\left| \frac{d^2}{d\lambda^2}m(\lambda) \right| = c \left| \frac{d}{d\lambda}\left( \lambda \lambda^{-\frac{d}{2}} J_{\frac{d}{2}}(2 \pi \lambda) \right) \right| \\
& \lesssim \lambda^{-\frac{d}{2}} \left| J_{\frac{d}{2}}(2 \pi \lambda) \right| + \lambda^{-\frac{d-2}{2}} \left|J_{\frac{d+2}{2}}(2 \pi \lambda) \right| \\
& \lesssim \lambda^{-\frac{d-1}{2}}.
\end{align*}
In the same manner, one inductively shows, using \eqref{equ-Bessel-2} and \eqref{equ-Bessel-1} that for $\lambda \geq \frac12$ and $k \in \N$, we have
$ \left| \frac{d^k}{d \lambda^k}m(\lambda)\right| \lesssim \lambda^{- \frac{d-1}{2}}$.
Therefore
\begin{align*}
\MoveEqLeft
\left(\int_{\frac12}^2 \left| \frac{d^k}{d\lambda^k}(m(2^n \cdot))|_\lambda \right|^2 d\lambda \right)^{\frac12} = 2^{nk} \left( \int_{\frac12}^2 \left|\frac{d^k}{d\lambda^k}m(2^n \lambda) \right|^2 d\lambda \right)^{\frac12} \\
& \lesssim 2^{nk} \left( \int_{\frac12}^2 |2^n\lambda|^{-2(\frac{d-1}{2})} d \lambda \right)^{\frac12} \\
& \cong 2^{nk} 2^{-n\frac{d-1}{2}},
\end{align*}
which is summable over $n \geq 0$ provided that $k < \frac{d-1}{2}$.
Moreover, by complex interpolation of $W^c_2(\R)$ spaces for non-integer $c > \frac12$ (see e.g. \cite[Corollary 3.11]{DK}), we obtain that for $c < \frac{d-1}{2}$, we have $\sum_{n \geq 0} \|m(2^n \cdot) \dyad_0\|_{W^c_2(\R)} < \infty$.

Now we want to apply Theorem \ref{thm-q-variation}.
To this end, we recall that the operator $A = \sqrt{- \Delta} \otimes \Id_Y$ has a $\Hor^\alpha_2$ calculus on $L^p(\R^d,Y)$ provided that
\[ \alpha > d \cdot \alpha_0(p,Y) + \frac12 , \]
where $\alpha_0(p,Y) = \max \left( \left| \frac1p - \frac12 \right|, \left| \frac{1}{p_Y} - \frac12 \right|, \left| \frac{1}{q_Y} - \frac12 \right| \right) \in [0,\frac12)$, $p_Y \in (1,2]$ and $q_Y \in [2,\infty)$ stand for the non-trivial convexity and concavity exponents of the UMD lattice $Y$.
See \cite[Remark 4.24]{DKK} first with $A = -\Delta$, and then noting (with $\beta = \frac12$) that $A^\beta$ has the same $\Hor^\alpha_2$ calculus as $A$ for any $\beta > 0$, since $m \mapsto m((\cdot)^\beta)$ is an isomorphism $\Hor^\alpha_2 \to \Hor^\alpha_2$.
Moreover, $\exp(-t\sqrt{-\Delta})$ is well-known to be lattice positive and contractive on $L^p(\R^d)$.
Combining the constraint on $c$ above coming from the spectral multiplier, the constraint on $\alpha$ above and the constraint $c >  \alpha + \max \left(\frac12, \frac{1}{\type L^p(Y)} - \frac{1}{\cotype L^p(Y)}\right) + 1$ from Theorem \ref{thm-q-variation}, we deduce from Theorem \ref{thm-q-variation} the finite $q$-variation
\[ \| t \mapsto A_tf\|_{L^p(\R^d,Y(V^q))} = \| t \mapsto m(t\sqrt{-\Delta})f\|_{L^p(\R^d,Y(V^q))} \lesssim \|f\|_{L^p(\R^d,Y)}, \]
provided that we have
\begin{align*}
c & < \frac{d-1}{2},\\
\alpha & > d \alpha_0(p,Y) + \frac12, \\
c & > \alpha + \max \left(\frac12, \frac{1}{\type L^p(Y)} - \frac{1}{\cotype L^p(Y)}\right) + 1.
\end{align*}
The combination of the three is possible if $\frac{d-1}{2}  > d \alpha_0(p,Y) + \frac12 + \max\left(\frac12, \frac{1}{\type L^p(Y)} - \frac{1}{\cotype L^p(Y)}\right) + 1$ so if $d \cdot \left( \frac12 - \alpha_0(p,Y) \right) > \frac12 + \frac12 + \max \left(\frac12, \frac{1}{\type L^p(Y)} - \frac{1}{\cotype L^p(Y)}\right) + 1$, in other words if
\[ d > d_0 := \frac{2 + \max\left(\frac12, \frac{1}{\type L^p(Y)} - \frac{1}{\cotype L^p(Y)}\right)}{\frac12 - \max\left( \left| \frac1p - \frac12 \right|, \left| \frac{1}{p_Y} - \frac12 \right|, \left| \frac{1}{q_Y} - \frac12 \right| \right)} . \]
Note that the denominator of $d_0$ is strictly positive since $Y$ has non-trivial concavity and convexity exponents, being a UMD lattice \cite[Corollary p.~216]{RdF}, \cite[p.~218-219]{TJ}.
\end{proof}

\begin{remark}
\label{rem-ball-means}
Theorem \ref{thm-spherical-means} immediately extends to the same result for the ball means $M_t$ from \eqref{equ-def-ball-mean} in place of the spherical means $A_t$, with the same constraint on the minimal dimension $d_0$.
That is, for $d > d_0$ with $d_0$ from \eqref{equ-minimal-dimension}, we have for any $f \in S(\R^d) \otimes Y$ with a constant $C_{p,q,d,Y}$ independent of $f$ that
\[ \|t \mapsto M_tf\|_{L^p(\R^d(Y(V^q)))} \leq C_{p,q,d,Y} \|f\|_{L^p(\R^d,Y)} . \]
Indeed, by Fubini we can write $M_tf(x,\omega) = \frac{d}{t^d} \int_0^t r^{d-1} A_rf(x,\omega) dr$.
Thus, for $(x,\omega) \in \R^d \times \Omega'$ fixed, we have
\begin{align*}
\MoveEqLeft
\left\| t \mapsto M_tf(x,\omega)\right\|_{V^q} = d \left\| t \mapsto t^{-d} \int_0^t r^{d-1} A_rf(x,\omega) dr\right\|_{V^q} \\
& = d \left\|t \mapsto \int_0^t \left( \frac{r}{t} \right)^d A_rf(x,\omega) \frac{dr}{r} \right\|_{V^q} \\
& = d \left\|t \mapsto \int_0^1 r^d A_{rt}f(x,\omega) \frac{dr}{r} \right\|_{V^q} \\
& \leq d \int_0^1 r^d \left\|t \mapsto A_{rt}f(x,\omega) \right\|_{V^q} \frac{dr}{r} \\
& = d \int_0^1 r^d \left\|t \mapsto A_tf(x,\omega)\right\|_{V^q} \frac{dr}{r} \\
& = d \frac{1}{d} \left\|t \mapsto A_tf(x,\omega)\right\|_{V^q}.
\end{align*}
Here we have used in the last but one step that the $V^q$ norm is dilation invariant, that is, $\|t \mapsto a_{rt}\|_{V^q} = \|t \mapsto a_t\|_{V^q}$ for $r > 0$.
This can be seen directly from the Definition \ref{def-q-variation} of the $q$-variation.
To conclude, it suffices to take $L^p(\R^d,Y)$ norms in the above inequality and to apply Theorem \ref{thm-spherical-means}.
\end{remark}

At the end of this section, we reprove \eqref{equ-HHL} for large dimensions, by the alternative proof of using our $q$-variational H\"ormander Functional Calculus Theorem \ref{thm-q-variation}.

\begin{prop}
\label{prop-HHL-alternative}
Let $Y$ be a UMD lattice.
Let $2 < q < \infty$, $1 < p < \infty$ and $d > d_0$ from \eqref{equ-minimal-dimension}.
Then there exists a constant $C_{p,q,Y}$ such that for any $f \in S(\R^d) \otimes Y$, we have
\[ \|t \mapsto M_t f\|_{L^p(\R^d,Y(V^q))} \leq C_{p,q,Y} \|f\|_{L^p(\R^d,Y)} . \]
\end{prop}

\begin{proof}
The proof is based on the dimension \emph{dependent} estimate of the $q$-variation of $t \mapsto A_t$ from Theorem \ref{thm-spherical-means}, together with the method of descent in the spirit of the Calder\'on-Zygmund method of rotations that we already employed in \cite{DeKr2}.
We introduce the following further mean operators for $d',d,k \in \N$, $t > 0$, $\theta \in O(d)$ an orthogonal mapping and $f \in S(\R^d)$:
\begin{align*}
A_{d,k,t}f(x) & = \frac{\int_{|y| \leq t} f(x-y) |y|^k dy}{\int_{|y| \leq t} |y|^k dy} \\
A_{d',t}^\theta f(x) & = \frac{\int_{|y_{d'}| \leq t} f(x - \theta(y_{d'},0)) |y_{d'}|^{d-d'} dy_{d'}}{\int_{|y_{d'}| \leq t} |y_{d'}|^{d-d'} dy_{d'}}.
\end{align*}
Here we write $y = (y_{d'},y_{d-d'}) \in \R^d = \R^{d'} \times \R^{d-d'}$.
Again, $A_{d,k,t}$ and $A_{d',t}^\theta$ admit obvious extensions to $f \in S(\R^d) \otimes Y$.
Then, for such an $f$ and $(x,\omega) \in \R^d \times \Omega'$ fixed, we obtain similarly to Remark \ref{rem-ball-means}
\begin{align}
\MoveEqLeft
\| t \mapsto A_{d,k,t}f(x,\omega) \|_{V^q} = c_d \left\| t\mapsto t^{-(d+k)} \int_0^t r^{d-1} r^k A_r f(x,\omega) dr \right\|_{V^q} \nonumber \\
& = (d+k) \left\| t \mapsto \int_0^t \left(\frac{r}{t}\right)^{d+k} A_r f(x,\omega) \frac{dr}{r} \right\|_{V^q} \nonumber \\
& = (d+k) \left\| t \mapsto \int_0^1 r^{d+k} A_{rt} f(x,\omega) \frac{dr}{r} \right\|_{V^q} \nonumber \\
& \leq (d+k) \int_0^1 r^{d+k} \left\|t \mapsto A_{rt} f(x,\omega) \right\|_{V^q} \frac{dr}{r} \nonumber \\
& \leq \left\|t \mapsto A_tf(x,\omega)\right\|_{V^q}. \label{equ-3-proof-prop-HHL-alternative}
\end{align}
Next we show that for $f \in S(\R^d)$,
\begin{equation}
\label{equ-1-proof-prop-HHL-alternative}
A_{d',t}^\Id f(x) = A_{d',d-d',t}\left(f  \circ g_{x_{d-d'}}\right)(x_{d'}) ,
\end{equation}
where
\[ g_{x_{d-d'}} : \R^{d'} \to \R^d, \: x_{d'} \mapsto (x_{d'},x_{d-d'}) . \]
We have
\begin{align*}
A_{d',t}^\Id f(x) & = \frac{\int_{|y_{d'}| \leq t} f(x - (y_{d'},0)) |y_{d'}|^{d-d'} dy_{d'}}{\int_{|y_{d'}| \leq t} |y_{d'}|^{d-d'} dy_{d'}} \\
& =  \frac{\int_{|y_{d'}| \leq t} f \circ g_{x_{d-d'}}(x_{d'} - y_{d'}) |y_{d'}|^{d-d'} dy_{d'}}{\int_{|y_{d'}| \leq t} |y_{d'}|^{d-d'} dy_{d'}} \\
& = A_{d',d-d',t} \left(f \circ g_{x_{d-d'}}\right)(x_{d'}),
\end{align*}
which shows \eqref{equ-1-proof-prop-HHL-alternative}.
We further note from \cite[(7)]{DeKr2},
\begin{equation}
\label{equ-2-proof-prop-HHL-alternative}
M_tf(x) = \int_{O(d)} A_{d',t}^\theta f(x)d\mu(\theta) ,
\end{equation}
where $d\mu$ denotes the normalized Haar measure of the compact group $O(d)$.
Thus, we obtain for $f \in S(\R^d) \otimes Y$, where $A_t$ denotes this time the spherical mean from \eqref{equ-def-spherical-mean}, but in dimension $d' < d$,
\begin{align*}
\MoveEqLeft
\left\| \| t \mapsto M_tf(x,\omega) \|_{V^q} \right\|_{L^p(\R^d,Y)} \overset{\eqref{equ-2-proof-prop-HHL-alternative}}{=} \left\| \left\| t \mapsto \int_{O(d)} A_{d',t}^\theta f(x,\omega) d\mu(\theta) \right\|_{V^q} \right\|_{L^p(\R^d,Y)} \\
& = \left\| \left\| t \mapsto \int_{O(d)} A_{d',t}^{\Id} \underset{=: \tilde{f}}{\underbrace{(f \circ \theta)}} \circ \theta^{-1}(x,\omega) d\mu(\theta) \right\|_{V^q} \right\|_{L^p(\R^d,Y)}\\
& \leq \int_{O(d)} \left\| \left\| t \mapsto A_{d',t}^{\Id} \tilde{f}(\cdot,\omega) \right\|_{V^q} \circ \theta^{-1} \right\|_{L^p(\R^d,Y)} d\mu(\theta) \\
& \overset{\eqref{equ-1-proof-prop-HHL-alternative}}{=} \int_{O(d)} \left\| \left\| \left\| t \mapsto A_{d',d-d',t} \left(\tilde{f} \circ g_{x_{d-d'}}\right) (x_{d'},\omega) \right\|_{V^q} \right\|_{L^p(\R^{d'},Y)} \right\|_{L^p(\R^{d-d'})} d\mu(\theta) \\
& \overset{\eqref{equ-3-proof-prop-HHL-alternative}}{\leq} \int_{O(d)} \left\| \left\| \left\| t \mapsto A_t \left(\tilde{f} \circ g_{x_{d-d'}}\right)(x_{d'},\omega) \right\|_{V^q} \right\|_{L^p(\R^{d'},Y)} \right\|_{L^p(\R^{d-d'})} d\mu(\theta) \\
& \overset{\mathrm{Theorem}\, \ref{thm-spherical-means}}{\leq} C_{p,q,d',Y}\int_{O(d)} \left\| \left\| \left( \tilde{f} \circ g_{x_{d-d'}} \right) (x_{d'},\omega) \right\|_{L^p(\R^{d'},Y)} \right\|_{L^p(\R^{d-d'})} d\mu(\theta) \\
& = C_{p,q,d',Y}\int_{O(d)} \|\tilde{f}\|_{L^p(\R^d,Y)} d\mu(\theta) \\
& = C_{p,q,d',Y}\|f\|_{L^p(\R^d,Y)}.
\end{align*}
Here we have applied Theorem \ref{thm-spherical-means} with dimension $d' > d_0$ from \eqref{equ-minimal-dimension}.
Thus the proposition follows for $d \geq d' > d_0$.
\end{proof}

\section{Schr\"odinger and wave maximal estimates}
\label{sec-wave}

In the next two sections, we will prove maximal estimates for the operator $\exp(itA)$, where $A$ is as before a semigroup generator acting on $L^p(Y)$ with a H\"ormander functional calculus.
As an application, we shall obtain pointwise continuity of the function $t \mapsto \exp(itA)f(x,\omega)$ with $f$ belonging to the domain of a fractional power of $A$ inside $L^p(Y)$.
The guiding example is Carleson's problem \cite[Theorem p.24]{Car} on the pointwise convergence of the solution of the free Schr\"odinger equation to the initial data, 
\[\exp(it\Delta)f(x) \to f(x)\text{ as }t \to 0,\:\text{a.e. }x \in \R^d. \]
Carleson asked what is the optimal parameter $s > 0$ such that $f$ belonging to the Sobolev space $H^s(\R^d) = \{ f \in L^2(\R^d) : \: \hat{f}(\xi) (1 + |\xi|^2)^{s/2} \in L^2(\R^d) \}$ guarantees that this pointwise convergence holds true.
A key point to prove pointwise convergence in Carleson's problem is the maximal estimate 
\[ \norm{ \sup_{0 < t < 1} |\exp(it\Delta)f| }_{L^2(B(0,1))} \leq C_s \norm{f}_{H^s(\R^d)}. \]
Later on, variants of this problem and the maximal estimate, also for the wave equation with $A = \sqrt{-\Delta}$ were studied by many authors, see e.g. \cite{RV1,RV2,Rog,BeGo,DeGu}.
Carleson found for his problem that in the one-dimensional case, $s > \frac14$ is sufficient and no $s < \frac18$ can be sufficient.
Later on, in particular in the last 10 years, many mathematicians contributed to the solution of Carleson's problem \cite{DaKe,Sjo,Veg,TVV,Lee,Bou13,Bou16,DGL,LuRo,DGLZ} (the parameter $s > \frac{d}{2(d+1)}$ is now known to be sufficient and no $s < \frac{d}{2(d+1)}$ is sufficient).
This has also opened the way to new problems such as to describe the Hausdorff dimension of the set of divergence in case of small values of $s$, which is beyond the scope of the present article.
In the following two sections we will contribute to both the maximal estimate $\norm{ \sup_t |\exp(itA)f|}_{L^p(Y)} < \infty$ and the pointwise convergence of $\exp(itA)f(x) \to f(x)$ as $t \to 0$,
in the quite general setting of operators $A$ as in Section \ref{sec-q-variation} before.

First, in this section, we show that Theorem \ref{thm-main-ft} below, which is a variant of Theorem \ref{thm-q-variation} gives to some extent maximal estimates for Schr\"odinger and wave operators, under rather general conditions on $A$, see Corollary \ref{cor-wave}.
The underlying question is the following.
Suppose that $A = - \Delta$ is the usual Laplacian operator on $L^p(\R^d)$ or that $A = \sqrt{-\Delta}$.
Then $t \mapsto u(t) = \exp(it\Delta)f$ is the solution of the Schr\"odinger equation with initial data $f \in L^p(\R^d)$, so
\[ \begin{cases} i \frac{\partial}{\partial t} u(t) & = - \Delta u(t) \\
u(0) & = f \end{cases}. \]
In the same way, $t \mapsto u(t) = \exp(it\sqrt{-\Delta})f$ is the solution of the wave equation with initial data $u(0) = f \in L^p(\R^d)$ and $u'(0) = i\sqrt{-\Delta} f$, so
\[ \begin{cases} \frac{\partial^2}{\partial t^2} u(t) & = \Delta u(t) \\
u(0) & = f \\
u'(0) & = i\sqrt{-\Delta} f
\end{cases}. \]
In \cite{RV1}, the authors investigate the boundedness of the associated Schr\"odinger maximal operators
\[ S^*f = \sup_{0 < t < 1}\left| \exp(it\Delta) f \right| \quad \text{and} \quad S^{**}f = \sup_{t \in \R} \left| \exp(it\Delta)f \right| . \]
Moreover, in \cite[(4),(5)]{RV2}, these authors study the boundedness of the wave maximal operators
\[ S^*f = \sup_{0 < t < 1}\left| \exp(\pm it\sqrt{-\Delta}) f \right| \quad \text{and} \quad S^{**}f = \sup_{t \in \R} \left| \exp(it\sqrt{-\Delta})f \right| . \]
Here, typical bounds are between spaces $S^*, S^{**} : H^s(\R^d) \to L^p(\R^d)$ \cite{RV1,RV2}.
For further results on maximal estimates of Schr\"odinger and wave operators, we refer to \cite{Rog,BeGo,DeGu}.
Moreover, the Sobolev space admits a description $H^s(\R^d) = \{ f \in L^2(\R^d), \: (1 - \Delta)^{\frac{s}{2}}f \in L^2(\R^d) \}$, so that the above maximal operator estimates read as follows
\begin{align}
\left\| \sup_{0 < t < 1} |\exp(\pm it A)f| \, \right\|_{L^p(\R^d)} & \lesssim \|(1 + A)^{s} f\|_{L^2(\R^d)} \label{equ-1-Schrodinger-wave} \\
\left\| \sup_{t \in \R}|\exp(it A)f| \, \right\|_{L^p(\R^d)} & \lesssim \|(1 + A)^{s} f\|_{L^2(\R^d)}, \label{equ-2-Schrodinger-wave}
\end{align}
for $A = - \Delta$ and $A = \sqrt{-\Delta}$, the right choices of $s$ and $p$ depending on the dimension $d$.
In this section, we shall prove an analogue of \eqref{equ-1-Schrodinger-wave} and \eqref{equ-2-Schrodinger-wave} for semigroup generators $A = A_0 \otimes \Id_Y$ such that the semigroup of $\exp(-tA_0)$ is positive and contractive on $L^p$, and such that $A$ has a H\"ormander calculus on $L^p(Y)$.
See Corollary \ref{cor-wave}.
Note that it is well-known that $A = -\Delta$ and thus also $A = \sqrt{-\Delta}$ do generate a positive contractive semigroup on $L^p(\R^d)$ and do have a H\"ormander $\Hor^\alpha_2$ functional calculus on $L^p(\R^d)$ for $\alpha > \frac{d}{2}$ \cite{Hor}.
Thus our quite general assumptions on $A$ are in particular satisfied in the above examples from \cite{RV1,RV2}.
In our framework, the norm on the left hand side of \eqref{equ-1-Schrodinger-wave} and \eqref{equ-2-Schrodinger-wave} will be in $L^p(Y)$.
Also it appears more natural that the $L^2$ space on the right hand side is replaced by $L^p(Y)$.
Our Sobolev space will be abstractly modelled as a domain of a power of $A$ as the right hand sides of \eqref{equ-1-Schrodinger-wave} and \eqref{equ-2-Schrodinger-wave} indicate.
The price we have to pay for this generality is that instead of suprema over $0< t < 1$ or $t \in \R$ as in \eqref{equ-1-Schrodinger-wave} or \eqref{equ-2-Schrodinger-wave}, we can only handle $t \in [a,b] \subseteq (0,\infty)$ compactly (via a test function $\psi_0 \in C^\infty_c(\R_+)$).
Also our parameter $\delta$ in \eqref{equ-3-Schrodinger-wave} below will be higher (that is, worse) than the parameter $s$ that obtained \cite{RV1,RV2} in \eqref{equ-1-Schrodinger-wave} and  \eqref{equ-2-Schrodinger-wave}.
We then shall get in Corollary \ref{cor-wave}
\begin{equation}
\label{equ-3-Schrodinger-wave}
\left\| \sup_{t > 0} \left|  \psi_0(t) \exp(itA) f \right| \, \right\|_{L^p(Y)} \lesssim_{\psi_0} \norm{(1+A)^\delta f}_{L^p(Y)} \quad (f \in D(A^\delta))
\end{equation}
for the correct choice of $\delta$ depending on the H\"ormander calculus order of $A$ and the geometry of $Y$.
The constant of the estimate depends on $\psi_0$, in particular its support.

The following Theorem is a variant of \cite[Theorem 3.1]{DK}, where elements $f \in L^p(Y)$ are replaced by functions $f_t \in L^p(Y(\Lambda^\beta))$.
Compared to \cite[Theorem 3.1]{DK}, it has a different value for the derivation exponent $c$ and involves the $W^c_1(\R)$ norm in place of the $W^c_2(\R)$ norm.
\
\begin{thm}
\label{thm-main-ft}
Let $Y$ be a UMD lattice, $1 < p < \infty$ and $(\Omega,\mu)$ a $\sigma$-finite measure space.
Let $\beta \geq 0$.
Let $A$ be a $0$-sectorial operator on $L^p(Y)$.
Assume that $A$ has a $\Hor^\alpha_2$ calculus on $L^p(Y)$ for some $\alpha > \frac12$.
Consider
\[ c > \alpha + \max\left(\frac12,\frac{1}{\type L^p(Y)} - \frac{1}{\cotype L^p(Y)}\right) + 1 +
 \beta. \]
Let $m$ be a spectral multiplier such that
\begin{enumerate}
\item $m \in W^{c}_1(\R)$ and $\supp (m) \subseteq [\frac12,2]$ or, more generally,
\item $m(0) = 0$, $m(2^n \cdot)\dyad_0 \in W^c_1(\R)$ for all $n \in \Z$ and $\sum_{n \in \Z} \|m(2^n \cdot)\dyad_0\|_{W^c_1(\R)} < \infty$, where $(\dyad_n)_{n \in \Z}$ is a partition of unity of $\R_+$.
\end{enumerate}
Then for $f_t \in L^p(Y(\Lambda^\beta_{2,2}(\R_+)))$, we have
\begin{equation}
\label{equ-1-prop-main-ft}
\| t \mapsto m(tA)f_t\|_{L^p(Y(\Lambda^\beta_{2,2}(\R_+)))} \leq C_m \| t \mapsto f_t\|_{L^p(Y(\Lambda^\beta_{2,2}(\R_+)))},
\end{equation}
where
\begin{enumerate}
\item $C_m \leq C \|m\|_{W^{c}_1(\R)}$,
\item $C_m \leq C \sum_{n \in \Z} \norm{m(2^n\cdot)\dyad_0}_{W^c_1(\R)}$.
\end{enumerate}
\end{thm}

\begin{proof}
The proof is a variant of that of Theorem \ref{thm-q-variation}.
We indicate the changements.

1. Let $\psi \in C^\infty_c(\R_+)$ as before and let also $\widetilde{\psi} = \sum_{k \in F} \dyad_k \in C^\infty_c(\R_+)$, where $F \subseteq \Z$ is finite, such that $\widetilde{\psi}(t) = 1$ for $t \in \supp(\psi)$.
Indeed, since $\supp(\psi)$ stays away from $0$ and $\infty$, such a funcion $\widetilde{\psi}$ exists.
As in the proof of Theorem \ref{thm-q-variation}, we obtain for $t \mapsto f_t \in L^p(Y(\Lambda^\beta))$, with,
\begin{align}
\MoveEqLeft
\| t \mapsto m(tA)f_t \|_{L^p(Y(\Lambda^\beta))} \cong \left\| \left( \sum_{n \in \Z} \|t \mapsto m(tA) \dyad_n(A) f_t \|_{\Lambda^\beta}^2 \right)^{\frac12} \right\|_{L^p(Y)} \nonumber \\
& = \left\| \left( \sum_{n \in \Z} \|t \mapsto \psi(t) \widetilde{\psi}(t) m(2^{-n}tA) \dyad_n(A) f_{2^{-n}t} \|_{\Lambda^\beta}^2 \right)^{\frac12} \right\|_{L p(Y)} \nonumber \\
& \lesssim R\left(\left\{ t \mapsto \psi(t) m(2^{-n}tA) : \: n \in \Z \right\}_{L^p(Y(\Lambda^\beta)) \to L^p(Y(\Lambda^\beta))} \right)  \left\| \left( \sum_{n \in \Z} \| t \mapsto \widetilde{\psi}(t) \dyad_n(A) f_{2^{-n} t} \|_{\Lambda^\beta}^2 \right)^{\frac12} \right\|_{L^p(Y)}. \label{equ-4-proof-prop-main-ft}
\end{align}
Note that the operator family subject to $R$-boundedness is now in $B(L^p(Y(\Lambda^\beta)))$.
Moreover, since $\{ \dyad_n(A) : \: n \in \Z \}_{L^p(Y(\Lambda^\beta)) \to L^p(Y(\Lambda^\beta))}$ is $R$-bounded (since $\|\dyad_n\|_{\Hor^\gamma_2} \leq C$ and $A \ot \Id_{\Lambda^\beta}$ has an $R$-bounded $\Hor^\gamma_2$ calculus), the square function in \eqref{equ-4-proof-prop-main-ft} above estimates further to
\begin{align*}
\MoveEqLeft
\ldots \leq \left\| \left( \sum_{n \in \Z} \|t \mapsto \widetilde{\psi}(t) f_{2^{-n}t} \|_{\Lambda^\beta}^2 \right)^{\frac12} \right\|_{L^p(Y)} \\
&  = \left\| \left( \sum_{n \in \Z} \|t \mapsto \widetilde{\psi}(2^n t) f_t \|_{\Lambda^\beta}^2 \right)^{\frac12} \right\|_{L^p(Y)} \\
&  \leq \sum_{k \in F} \left\| \left( \sum_{n \in \Z} \|t \mapsto \dyad_k(2^n t) f_t \|_{\Lambda^\beta}^2 \right)^{\frac12} \right\|_{L^p(Y)} \\
&  \cong \left\| \left( \sum_{n \in \Z} \|t \mapsto \dyad_n(t) f_t \|_{\Lambda^\beta}^2 \right)^{\frac12} \right\|_{L^p(Y)}.
\end{align*}
Now for fixed $(x,\omega) \in \Omega \times \Omega'$, we have
\begin{align*}
\MoveEqLeft
\sum_{n \in \Z} \|t \mapsto \dyad_n(t) f_t(x,\omega)\|_{\Lambda^\beta}^2 \cong \E \left\| \sum_{n \in \Z} \epsi_n \cdot (t \mapsto \dyad_n(t) f_t(x,\omega)) \right\|_{\Lambda^\beta}^2 \\
& \cong \| t \mapsto f_t(x,\omega) \|_{\Lambda^\beta}^2,
\end{align*}
where the last step follows from the Paley-Littlewood equivalence \cite[Theorem 4.1]{KrW2} for the $0$-sectorial operator $A(t \mapsto f_t) = t \mapsto t \cdot f_t$ on the space $X = \Lambda^\beta$ which indeed has a Mihlin calculus as needed in this source.
We infer that 
\begin{equation}
\label{equ-1-proof-prop-main-ft}
\left\| \left( \sum_{n \in \Z} \|t \mapsto \widetilde{\psi}(t) \dyad_n(A) f_{2^{-n} t} \|_{\Lambda^\beta}^2 \right)^{\frac12} \right\|_{L^p(Y)} \lesssim \|t \mapsto f_t\|_{L^p(Y(\Lambda^\beta))} .
\end{equation}
It remains to estimate the $R$-bound of
\begin{equation}
\label{equ-2-proof-prop-main-ft}
\left\{ t \mapsto \psi(t) m(2^{-n}tA) : \: n \in \Z \right\}_{L^p(Y(\Lambda^\beta)) \to L^p(Y(\Lambda^\beta))} .
\end{equation}
Develop $\psi(t)m(2^n t A)f$ as in \eqref{equ-develop-in-wave}.
Note that $\Lambda^\beta$ is a Hilbert space (isometric to $\ell^2$).
Note also that Rademacher sums in $L^p(Y)$ are equivalent to square sums in the sense of \eqref{equ-Rademacher-square}.
Then use that by Lemma \ref{lem-UMD-lattice-Hilbert-extension}, $A \otimes \Id_{\Lambda^\beta}$ has an $R$-bounded $\Hor^\gamma_2$ calculus on $L^p(Y(\Lambda^\beta))$ to get the $R$-bounded operator family $\{ (1 + 2^n A)^\gamma \phi(2^n A) : \: n \in \Z \}_{L^p(Y(\Lambda^\beta)) \to L^p(Y(\Lambda^\beta))}$.
Then as in the proof of Theorem \ref{thm-q-variation}, we are left to show that
\[ \left\{ \frac{1}{2\pi} \int_\R h_s \langle s \rangle^{-(\gamma + \delta)} ( 1+  2^n A)^{-\gamma} \exp(i2^n sA) ds :\: n \in \Z \right\}_{L^p(Y(\Lambda^\beta)) \to L^p(Y(\Lambda^\beta))} \]
is an $R$-bounded operator family, but this time as operators $L^p(Y(\Lambda^\beta)) \to L^p(Y(\Lambda^\beta))$.
Here $h_s \in \Lambda^\beta$ is the function from \eqref{equ-h_s}.
Write in short $T_s^n = \frac{1}{2\pi} \langle s \rangle^{-(\gamma + \delta)}(1 + 2^n A)^{-\gamma} \exp(i2^n s A)$.
Then for $f_1,\ldots,f_N \in L^p(Y(\Lambda^\beta))$ and any $n_1,\ldots,n_N \in \Z$,
\begin{align}
\MoveEqLeft
\left\| \left( \sum_{k = 1}^N \left\| \int_\R h_s(t) T_s^{n_k} f_k(t) ds \right\|_{\Lambda^\beta}^2 \right)^{\frac12} \right\|_{L^p(Y)} \leq \left\| \left( \sum_{k = 1}^N \left( \int_\R \|h_s(t) T_s^{n_k} f_k(t) \|_{\Lambda^\beta} ds \right)^2 \right)^{\frac12} \right\|_{L^p(Y)} \nonumber \\
& \leq \sup_{s \in \R } \|h_s\|_{M(\Lambda^\beta)} \left\| \left( \sum_{k = 1}^N \left( \int_\R \|T_s^{n_k} f_k(t)\|_{\Lambda^\beta} ds \right)^{2} \right)^{\frac12} \right\|_{L^p(Y)} \nonumber \\
& \leq \sup_{s \in \R } \|h_s\|_{M(\Lambda^\beta)} \int_\R \left\| \left( \sum_{k = 1}^N \|T_s^{n_k} f_k(t)\|_{\Lambda^\beta}^2 \right)^{\frac12} \right\|_{L^p(Y)} ds \nonumber \\
& \lesssim \sup_{s \in \R } \|h_s\|_{M(\Lambda^\beta)} \int_\R R \left( \left\{ T_s^n : \: n \in \Z \right\}_{L^p(Y(\Lambda^\beta)) \to L^p(Y(\Lambda^\beta))} \right) ds \left\| \left( \sum_{k = 1}^N \|f_k(t)\|_{\Lambda^\beta}^2 \right)^{\frac12} \right\|_{L^p(Y)}.
\label{equ-3-proof-prop-main-ft}
\end{align}

Here, $M(\Lambda^\beta)$ stands for the Banach algebra of pointwise multipliers of $\Lambda^\beta$.
According to \cite[p.~140]{Tri}, we have for the classical Sobolev space, $M(W^\beta_2(\R)) \supseteq C^{\beta'}$ for $\beta' > \beta$.
Note that since $\psi$ has compact support in $\R_+$, similarly to Remark \ref{rem-Lambda-dilation-invariant}, we have $\|h_{s,e}\|_{C^{\beta'}} \cong \|h_s\|_{C^{\beta'}}$, so that $\|h_s\|_{M(\Lambda^\beta)} \lesssim \|h_s\|_{C^{\beta'}}$.
Let us estimate $\|h_s\|_{C^{\beta'}} = \sup_{\alpha \leq \beta'} \|\partial^\alpha h_s\|_{\infty}$.
We have with $\xi(t) = t \psi\left( \frac{1}{t} \right) \in C^\infty_c(\R_+)$,
\begin{align}
\MoveEqLeft
\|h_s\|_{L^\infty(\R)} = \left\|\partial^{\beta'}\xi(t) \hat{m}\left(\frac{s}{t} \right) \right\|_{L^\infty_t(\R)} \langle s \rangle^{\gamma + \delta} \lesssim \sup_{t \in \left[\frac18,8\right]} \left|\hat{m}\left(\frac{s}{t}\right)\right| \langle s \rangle^{\gamma + \delta} \label{equ-5-proof-prop-main-ft}\\
& \leq \sup_{s \in \R} \sup_{t \in \left[\frac18,8\right]} \left| \hat{m}(s) \right| \langle s t \rangle^{\gamma + \delta} \nonumber \\
& \lesssim \left\|\hat{m}(s) \langle s \rangle^{\gamma + \delta} \right\|_{L^\infty_s(\R)} \lesssim \|m\|_{W^{\gamma + \delta}_1(\R)} .\nonumber
\end{align}
Moreover
\begin{align*}
\MoveEqLeft
\left\|\xi(t) \partial^{\beta'} \hat{m}\left(\frac{s}{t} \right) \right\|_{L^\infty_t(\R)} \langle s \rangle^{\gamma + \delta} \lesssim \left\| \partial^{\beta'}\left( \hat{m} \right) (s) s^{\beta'} \langle s \rangle^{\gamma + \delta} \right\|_{L^\infty_s(\R)} \\
& \lesssim \left\| x^{\beta'} m(x) \right\|_{W^{\gamma + \delta + \beta'}_1(\R)} \lesssim \|m\|_{W^{\gamma + \delta + \beta'}_1(\R)},
\end{align*}
where the last step follows from the fact that $m$ has compact support in $\R_+$, similarly to Remark \ref{rem-Lambda-dilation-invariant}.
Thus, 
\[ \sup_{s \in \R} \|h_s\|_{M(\Lambda^\beta)} \lesssim \sup_{s \in \R} \|h_s\|_{C^{\beta'}} \lesssim \|m\|_{W^{\gamma + \delta + \beta'}_1(\R) } < \infty . \]
To obtain the bound in \eqref{equ-3-proof-prop-main-ft}, it suffices to note that according to Lemma \ref{lem-UMD-lattice-Hilbert-extension}, 
\[ \int_{\R}  R \left( \left\{ T_s^n : \: n \in \Z \right\}_{L^p(Y(\Lambda^\beta)) \to L^p(Y(\Lambda^\beta))} \right) ds = \int_{\R}  R \left( \left\{ T_s^n : \: n \in \Z \right\}_{L^p(Y) \to L^p(Y)} \right) ds < \infty . \]
This shows with the calculation \eqref{equ-3-proof-prop-main-ft} that the set in \eqref{equ-2-proof-prop-main-ft} is $R$-bounded.
With the square function estimate \eqref{equ-1-proof-prop-main-ft}, we conclude the bound in \eqref{equ-4-proof-prop-main-ft}, and thus the proof of the first case.

2. Do as in the proof of Theorem \ref{thm-q-variation}, exploiting that $\|t \mapsto f_{2^n t}\|_{L^p(Y(\Lambda^\beta))} = \|t \mapsto f_t\|_{L^p(Y(\Lambda^\beta))}$.
\end{proof}

In the case that $m(0) \neq 0$, we have the following variant of Theorem \ref{thm-main-ft} in which we obtain an estimate of the pointwise supremum of $t \mapsto m(tA)f_t$.
Here the norm $\norm{m}_{\Hor^\alpha_1}$ is defined in a similar manner to $\norm{m}_{\Hor^\alpha_2}$ in Definition \ref{defi-Hoermander-class}, namely,
$\norm{m}_{\Hor^\alpha_1} = |m(0)| + \sup_{R > 0} \| \phi m(R \,\cdot) \|_{W^\alpha_1(\R)}$ for any $C^\infty_c(0,\infty)$ function $\phi$ different from the constant $0$ function.
Also, $\cutoff_0$ is as in Theorem \ref{thm-q-variation}, namely, $\cutoff_0 \in C^\infty(\R_+)$ with support in $(0,4]$ and equal to $1$ on $(0,2]$.

\begin{prop}
\label{prop-ft-exp}
Let $Y$ be a UMD lattice, $1 < p < \infty$ and $(\Omega,\mu)$ a $\sigma$-finite measure space.
Let $\beta > \frac12$.
Let $A$ be a $0$-sectorial operator on $L^p(Y)$.
Assume that $A$ has a $\Hor^\alpha_2$ calculus on $L^p(Y)$ for some $\alpha > \frac12$.
In this Proposition, we assume that $A$ is of the form $A = A_0 \otimes \Id_Y$ and that $\exp(-tA_0)$ is lattice positive and contractive on $L^p(\Omega)$.
Let $m$ be a spectral multiplier such that 
\[ m|_{[0,1]} \in C^1[0,1], \: \|m' \cutoff_0 \|_{\Hor^{c-1}_1} < \infty ,\: \sum_{n \geq 0} \|m(2^n \cdot) \dyad_0\|_{W^c_1(\R)} < \infty , \]
where
\[ c > \alpha + \max\left(\frac12,\frac{1}{\type L^p(Y)} - \frac{1}{\cotype L^p(Y)}\right) + 1 + \beta. \]
Then for $f_t \in L^p(Y(\Lambda^\beta_{2,2}(\R_+)))$, we have
\begin{align}
\label{equ-1-prop-ft-exp}
& \| t \mapsto m(tA)f_t\|_{L^p(Y(L^\infty(\R_+)))} \leq \\
& C \left( |m(0)| + \|m' \cutoff_0\|_{\Hor^{c-1}_1} + \sum_{n \geq 0} \|m(2^n \cdot) \dyad_0\|_{W^{c}_1(\R)} \right) \| t \mapsto f_t\|_{L^p(Y(\Lambda^\beta_{2,2}(\R_+)))}. \nonumber
\end{align}
\end{prop}

\begin{proof}
The proof is a variant of the proof of case 3 in Theorem \ref{thm-q-variation}.
We put $m_1(t) = m(t) - m(0) \exp(-t)$.
In particular, we note that 
\[ \left\|\sup_{t > 0} \left| e^{-tA} f_t\right| \, \right\|_{L^p(Y)} \leq \left\| \sup_{t > 0}e^{-tA} \sup_{s > 0} |f_s| \,\right\|_{L^p(Y)} \lesssim \left\| \sup_{s > 0} |f_s| \, \right\|_{L^p(Y)} \overset{\text{Lemma }\ref{lem-q-variation}}{\lesssim} \left\| f \right\|_{L^p(Y(\Lambda^\beta))} . \]
The rest of the proof is identical to that of case 3 in Theorem \ref{thm-q-variation}, replacing the integration exponent $2$ by $1$ everywhere.
\end{proof}

The main result of this section then reads as follows.

\begin{cor}
\label{cor-wave}
Let $Y$ be a UMD lattice, $1 < p < \infty$ and $(\Omega,\mu)$ a $\sigma$-finite measure space.
Let $A$ be a $0$-sectorial operator on $L^p(Y)$.
Assume that $A$ has a $\Hor^\alpha_2$ calculus on $L^p(Y)$ for some $\alpha > \frac12$.
In this corollary, we assume that $A$ is of the form $A = A_0 \otimes \Id_Y$ and that $\exp(-tA_0)$ is lattice positive and contractive on $L^p(\Omega)$.
Let
\[ \delta > \alpha + \max\left(\frac12, \frac{1}{\type L^p(Y)} - \frac{1}{\cotype L^p(Y)} \right) + \frac32 . \]
Let $\psi_0 \in C^\infty_c(\R_+)$, i.e. with support included in a compact interval $\subseteq (0,\infty)$.
Then, with implied constants depending on $\psi_0$ and its support, we have
\[
\| \sup_{t > 0} |\psi_0(t) \exp(itA) f| \, \|_{L^p(Y)} \lesssim \norm{(1+A)^\delta f}_{L^p(Y)} \quad (f \in D(A^\delta)).
\]
\end{cor}

\begin{proof}
We shall apply Proposition \ref{prop-ft-exp} to the spectral multiplier $m(\lambda) = (1 + \lambda)^{-\delta} \exp(i\lambda)$ and the element $f_t = \psi_0(t) (1 + tA)^{\delta} f$ for some $f \in D(A^\delta) \subseteq L^p(Y)$.
Let $\beta > \frac12$ such that $\delta > \alpha + \max\left(\frac12, \frac{1}{\type L^p(Y)} - \frac{1}{\cotype L^p(Y)} \right) + 1 + \beta$.
Note that according to \cite[Corollary 3.11]{DK}, $m$ satisfies the hypotheses of Proposition \ref{prop-ft-exp} for our choice of $\beta$ and $\delta$, since in general $\|m'\cutoff_0\|_{\Hor^{c-1}_1} \lesssim \|m'\cutoff_0\|_{\Hor^{c-1}_2}$ and $\|m(2^n\cdot)\dyad_0\|_{W^c_1(\R)} \lesssim_{\dyad_0} \|m(2^n \cdot) \dyad_0\|_{W^c_2(\R)}$.
Then we obtain
\begin{align*}
\MoveEqLeft
\left\| \sup_{t > 0} |\psi_0(t) \exp(itA)f| \, \right\|_{L^p(Y)} = \left\| \sup_{t > 0} \left| (1 + tA)^{-\delta} \exp(itA) \left( 1 + t A)^{\delta} \psi_0(t) f \right) \right| \, \right\|_{L^p(Y)} \\
& = \left\| \sup_{t > 0} \left| m(tA)f_t \right| \, \right\|_{L^p(Y)} \overset{\mathrm{Proposition} \, \ref{prop-ft-exp}}{\lesssim} \|f_t \|_{L^p(Y(\Lambda^\beta))} \\
& = \left\| t \mapsto \psi_0(t) (1 + tA)^{\delta} f \right\|_{L^p(Y(\Lambda^\beta))}.
\end{align*}
We estimate this further.
Since $\psi_0$ has compact support in $\R_+$, by Remark \ref{rem-Lambda-dilation-invariant}, we have
\begin{align*}
\MoveEqLeft
\left\| t \mapsto \psi_0(t) (1 + tA)^{\delta} f \right\|_{L^p(Y(\Lambda^\beta))} \lesssim  \left\| t \mapsto \psi_0(t) (1 + tA)^{\delta} f \right\|_{L^p(Y(W^{\lceil \beta \rceil}_2(\R)))} \\
& \lesssim \left\| t \mapsto \psi_0(t) (1 + tA)^{\delta} f\right\|_{L^p(Y(L^2(\R)))} + \sum_{k=1}^{\lceil \beta \rceil} \left\| t \mapsto \frac{d^{k}}{dt^{k}} \left( \psi_0(t) ( 1 + tA)^{\delta} f \right) \right\|_{L^p(Y(L^2(\R)))} \\
& \lesssim_{\psi_0} \left\| \sup_{t \in \supp(\psi_0)} \left|(1 + tA)^\delta f \right| \, \right\|_{L^p(Y)} + \max_{k = 0,\ldots,\lceil \beta \rceil}\left\| \sup_{t \in \supp(\psi_0)} \left| A^k (1 + tA)^{\delta- k} f \right| \, \right\|_{L^p(Y)}.
\end{align*}
We fix a $k \in \{ 0, 1, \ldots, \lceil \beta \rceil \}$.
Then 
\[ A^k  ( 1 + z A)^{\delta - k} f = A^k ( 1+ A)^{-\delta} (1 + z A)^{\delta - k } (1 + A)^\delta f = \phi_z(A) f_1\]
with $\phi_z(\lambda) = \lambda^k(1 + \lambda)^{-\delta} ( 1 + z \lambda)^{\delta - k}$
and $f_1 = (1 + A)^\delta f \in L^p(Y)$ since $f \in D(A^\delta)$.
Then $\phi_z$ is holomorphic and bounded on $\C_+$, thus $\phi_z(A)$ is a bounded operator thanks to the H\"ormander calculus, thus $\HI$ calculus of $A$.
Moreover, the function $z \mapsto \phi_z(A)f_1$ is analytic in some neighborhood of $\supp(\psi_0) \subseteq (0,\infty)$.
Thus, by the Cauchy integral formula,
\[ \phi_t(A) f_1 = \frac{1}{2\pi i} \int_\Gamma (z-t)^{-1} \phi_z(A)f_1 dz , \]
where $\Gamma$ is some simple contour surrounding $\supp(\psi_0)$.
We obtain 
\[\sup_{t \in \supp(\psi_0)} | A^k (1 + tA)^{\delta - k}f | \leq \frac{1}{2\pi} \left(\sup_{t \in \supp(\psi_0),z \in \Gamma} | (z -t)^{-1}| \right)  \int_\Gamma  |\phi_z(A)f_1| \: |dz| \]
and thus,
\begin{align*}
\MoveEqLeft
\norm{\sup_{t \in \supp(\psi_0)} | A^k ( 1 + t A)^{\delta - k}f | \: }_{L^p(Y)} \lesssim \int_\Gamma \norm{\phi_z(A)f_1}_{L^p(Y)} |dz| \\
& \lesssim \norm{f_1}_{L^p(Y)} = \norm{(1 + A)^\delta f}_{L^p(Y)} < \infty,
\end{align*}
where $\norm{\phi_z(A)f_1} \lesssim \norm{f_1}$ thanks to bounded $\HI$ calculus.
\end{proof}

\section{Pointwise continuity of time paths for abstract Schr\"odinger and wave equations}
\label{subsec-Carleson}

The goal of this section is to provide an application of Corollary \ref{cor-wave} to the continuity in time $t \in \R \mapsto u(t,x,\omega) = \exp(itA)f(x,\omega)$ of the solution of the abstract Schr\"odinger equation
\[ \begin{cases}
- i \frac{\partial}{\partial t} u(t,x,\omega) & = A_xu(t,x,\omega) \\
u(0,x,\omega) & =  f(x,\omega). \end{cases} \]
This requires naturally the inital data $f$ to be sufficiently smooth, i.e. belonging to a fractional domain space $D(A^\delta)$ of $A$.
Note that when $A = -\Delta$ on $L^2(\R^d)$ (classical Schr\"odinger equation), then Carleson \cite[Theorem p.24]{Car} asked for the optimal parameter $\delta$ such that $f \in H^\delta(\R^d) = D(A^{\delta/2})$ implies that $u(t,x) \to f(x)$ pointwise a.e. $x \in \R^d$ as $t \to 0+$ (which then yields by isometry of $\exp(itA)$ on $H^\delta(\R^d)$ that $u(t,x) \to u(t_0,x)$ pointwise a.e. $x \in \R^d$ as $t \to t_0$).
Carleson himself found that in the one-dimensional case, $\delta > \frac14$ is sufficient and no $\delta < \frac18$ is sufficient.
Du-Guth-Li \cite{DGL} recently found that in the two-dimensional case $\delta > \frac13$ is sufficient, and Du-Zhang \cite{DuZh} that in the $d$-dimensional case, $\delta > \frac{d}{2(d+1)}$ is sufficient.
That no $\delta < \frac{d}{2(d+1)}$ is admissible is in turn a recent result by Bourgain \cite{Bou16}.
In this section we provide an abstract version how to deduce from H\"ormander calculus of a $0$-sectorial operator such pointwise continuity of Schr\"odinger (or wave) equation paths.
Lebesgue exponent is as before in the range $1 < p < \infty$.
The price for our generality is that we need a higher value of $\delta$ than for the classical free Schr\"odinger equation.
However, our approach allows general measure spaces $\Omega$ in place of $\R^d$, more general operators $A$ and works also for $L^p(Y)$ valued initial values, lying in a domain of a fractional power of $A$.
The analogous question of pointwise continuity of the solution for the free wave equation, $u(t,x)= \exp(it \sqrt{-\Delta})f(x)$ seems to be more difficult and less results appear in the literature than for the Schr\"odinger equation. 
See e.g. \cite{RV2}.
Our approach works also for the wave equation, see e.g. Corollary \ref{cor-Zhang} with $a = 1$.

In the first main result, Theorem \ref{thm-continuous-paths}, we obtain continuity on $(0,\infty)$, i.e. outside of $0$.
Then we will enhance the result to continuity on $\R$ in Corollary \ref{cor-continuity-on-the-real-line}.
At the end we compare our result to others before in the literature and thus return to the case of $A$ being selfadjoint on some $L^2$ space, where the parameter $\delta$ simplifies and our space $D(A^\delta)$ often becomes again a Sobolev $H^s(\R^d)$ space.

\begin{thm}
\label{thm-continuous-paths}
Let $Y$ be a UMD lattice, $1 < p < \infty$ and $(\Omega,\mu)$ a $\sigma$-finite measure space.
Let $A = A_0 \ot \Id_Y$ have a $\Hor^\alpha_2$ calculus on $L^p(Y)$ for some $\alpha > \frac12$, and $A_0$ generate a lattice positive and contractive semigroup on $L^p(\Omega)$.
Pick
\[ \delta > \alpha + \max \left( \frac12, \frac{1}{\type L^p(Y)} - \frac{1}{\cotype L^p(Y)} \right) + \frac{3}{2} . \]
Then for any $f \in D(A^\delta)$, for a.e. $(x,\omega) \in \Omega \times \Omega'$, the function
\[ \begin{cases} (0,\infty) & \longmapsto \C \\ t & \longmapsto \exp(itA)f(x,\omega) \end{cases} \]
is continuous.
Similarly, for a.e. $x \in\Omega$, the function
\[ \begin{cases} (0,\infty) & \longmapsto Y \\ t & \longmapsto \exp(itA)f(x,\cdot) \end{cases} \]
is continuous.
\end{thm}

The proof is divised in several steps which we address in the following lemmas.
Recall the calculus kernel $D_A= \{ \phi(A)f :\: f \in L^p(Y),\: \phi \in C^\infty_c(\R_+) \}$ from Lemma \ref{lem-representation-formula-wave-operators}.
We fix in the following a function $\psi_0 \in C^\infty_c(0,\infty)$.

\begin{lemma}
\label{lem-density-calculus-kernel-in-H}
Assume that the assumptions of Theorem \ref{thm-continuous-paths} hold.
The calculus kernel $D_A$ is contained in $D(A^\delta)$ and is dense in it with respect to the norm $\norm{(1+A)^\delta f}_{L^p(Y)}$.
\end{lemma}

\begin{proof}
Let $f \in D_A$ and $\phi \in C^\infty_c(\R_+)$ such that $f = \phi(A)f$.
Then $m(\lambda) = \phi(\lambda) (1 + \lambda)^\delta$ belongs to $C^\infty_c(\R_+)$, and thus, $m(A)f$ belongs to $L^p(Y)$.
But $m(A)f = (1+A)^\delta \phi(A)f= (1+A)^\delta f$, so $f \in D(A^\delta)$.
Moreover, according to the Convergence Lemma \ref{lem-Hormander-convergence-lemma}, if $g \in D(A^\delta)$ and $(\dyad_n)_{n \in \Z}$ is a dyadic partition of unity of $\R_+$,
$f_n = \sum_{k=-n}^n \dyad_k(A)g \in D_A$ and
\[ \norm{(1+A)^\delta(f_n-g)}_{L^p(Y)} = \norm{\left(\sum_{k=-n}^n \dyad_k(A) - \Id \right)(1+A)^\delta g}_{L^p(Y)} \to 0 \quad (n \to \infty) . \]
\end{proof}

\begin{lemma}
\label{lem-calculus-kernel-holomorphic-extension}
Assume that the assumptions of Theorem \ref{thm-continuous-paths} hold.
Let $f \in D_A$ belong to the calculus kernel.
Then for a.e. $(x,\omega) \in \Omega \times \Omega'$, the mapping $\R_+ \to \C, \: t \mapsto \exp(itA)f(x,\omega)$ extends to an entire holomorphic function on $\C$.
In particular, it is a continuous function.
Similarly, for a.e. $x \in \Omega$ the mapping $\R_+ \to Y,\: t \mapsto \exp(itA)f(x,\cdot)$ extends to an entire function, thus continuous on $\R_+$.
Finally, $\C \times \Omega \times \Omega' \to \C, \: (z,x,\omega) \mapsto \exp(izA)f(x,\omega)$ is a measurable function.
\end{lemma}

\begin{proof}
Let $\phi \in C^\infty_c(0,\infty)$ such that $f = \phi(A)f$.
The function $z \mapsto \exp(iz\lambda) \phi(\lambda)$ is holomorphic $\C \to \Hor^\alpha_2$, so the function $z \mapsto \exp(izA) \phi(A)$ is holomorphic $\C \to B(L^p(Y))$.
Thus there exists a sequence $(f_n)_n$ in $L^p(Y)$ such that 
\[ F(z) = \exp(izA) \phi(A)f = \sum_{n = 0}^\infty f_n z^n \quad (z \in \C) .\]
By the Cauchy integral formula, for any $R > 0$ there exists $C_R < \infty$ such that $\norm{f_n} \leq C_R R^{-n}$ for any $n \in \N_0$.
For $N \in \N$ and $R > 0$, we consider the truncation 
\[ F_{N,R}(z) = \sum_{n = 0}^\infty f_n 1_{|f_n| \leq N R^{-n}} z^n . \]
Then $F_{N,R}$ is holomorphic $\C \to L^p(Y)$, but also 
\[ z \mapsto F_{N,R}(z,x,\omega) = \sum_{n=0}^\infty f_n(x,\omega) 1_{|f_n(x,\omega)| \leq N R^{-n}} z^n \]
is holomorphic $B(0,R) \to \C$ for fixed $(x,\omega) \in \Omega \times \Omega'$.
For $n \in \N_0$ fixed, $R^n \norm{f_n 1_{|f_n| > N R^{-n}}}_{L^p(Y)} = \norm{R^n f_n 1_{|R^n f_n| > N}}_{L^p(Y)} \to 0$ as $N \to \infty$ thanks to the $\sigma$-order continuity of $L^p(Y)$.
Moreover $\norm{R^n f_n 1_{|R^n f_n| > N}}_{L^p(Y)} \leq R^n \norm{f_n}_{L^p(Y)} \leq C_R$ is bounded for $n,N \in \N_0$.
Then we deduce
\begin{align*}
\MoveEqLeft
\norm{\sup_{z \in B(0,R - \epsi)} |F_{N,R}(z) - F(z)| \: }_{L^p(Y)}  \leq \sum_{n =0}^\infty \norm{R^nf_n 1_{|f_n| > N R^{-n}}}_{L^p(Y)} \frac{(R-\epsi)^n}{R^n} \\
 & \longrightarrow  0 \quad (N \to \infty).
\end{align*}
According to a standard proof of completeness of $L^p$, convergence in $L^p(Y)$ implies pointwise a.e. convergence of a subsequence.
Thus, for a.e. $(x,\omega) \in \Omega \times \Omega'$, $F(z)$ is the uniform limit of $F_{N_k,R}$.
Since analyticity is preserved under uniform limits, we infer that $z \mapsto F(z,x,\omega)$ is indeed analytic $B(0,R - \epsi) \to \C$ for a.e. $(x,\omega)$.
Letting $R \to \infty$ then shows the claim on analyticity.
Since $F_{N,R}$ is clearly measurable as a function in $(z,x,\omega)$ as pointwise limit of easy functions, also the uniform (hence pointwise) limit $F$ is measurable as a function in $(z,x,\omega)$.
\end{proof}

\begin{proof}[of Theorem \ref{thm-continuous-paths}]
Let $f \in D(A^\delta)$ and choose $g \in D_A$ an approximation of $f$ as in Lemma \ref{lem-density-calculus-kernel-in-H}.
Assume that $\psi_0 \in C^\infty_c(0,\infty)$ satisfies $\psi_0(t) = 1$ for $t$ in a neighborhood of $[a,b]$, where $[a,b] \subseteq (0,\infty)$ is any fixed interval on which we want to prove continuity of paths.
We have
\begin{align}
\MoveEqLeft
\norm{ \limsup_{\epsi \to 0} \sup_{|t-t_0| < \epsi, \: t_0 \in [a,b]}|\exp(itA)f - \exp(it_0A)f| \: }_{L^p(Y)} \nonumber \\
& \leq \norm{ \limsup_{\epsi \to 0} \sup_{|t-t_0| < \epsi, \: t_0 \in [a,b]} |\exp(itA)(f-g)| \:}_{L^p(Y)} \nonumber \\
& + \norm{ \limsup_{\epsi \to 0} \sup_{|t-t_0| < \epsi ,\: t_0 \in [a,b]} |\exp(itA)g - \exp(it_0A)g| \:}_{L^p(Y)} \nonumber \\
& + \norm{ \sup_{t_0 \in [a,b]} |\exp(it_0A)(f-g)| \:}_{L^p(Y)}. \label{equ-1-proof-thm-continuous-paths}
\end{align}
For the first expression in \eqref{equ-1-proof-thm-continuous-paths}, we note that
\begin{align*}
\MoveEqLeft
\limsup_{\epsi \to 0} \sup_{|t-t_0| < \epsi, \: t_0 \in [a,b]} |\exp(itA)(f-g)(x,\omega)| \leq \sup_{t \in [a - \epsi,b+\epsi]} | \exp(itA)(f-g)(x,\omega) | \\
& \leq \sup_{t \in (0,\infty)} |\psi_0(t) \exp(itA)(f-g)(x,\omega)|.
\end{align*}
Therefore,
\begin{align*}
\MoveEqLeft
\norm{ \limsup_{\epsi \to 0} \sup_{|t-t_0| < \epsi, \: t_0 \in [a,b]} |\exp(itA)(f-g)|}_{L^p(Y)} \\
& \leq  \norm{\sup_{t \in (0,\infty)} | \psi_0(t) \exp(itA)(f-g)| \: }_{L^p(Y)} \\
& \overset{Corollary \ref{cor-wave}}{\lesssim} \norm{ (1+A)^\delta (f-g)}_{L^p(Y)}
\end{align*}
The same estimate applies for the third expression in \eqref{equ-1-proof-thm-continuous-paths} and shows that
\[ \norm{ \sup_{t_0 \in [a,b]} |\exp(it_0A)(f-g)| \: }_{L^p(Y)} \lesssim \norm{(1+A)^\delta (f-g)}_{L^p(Y)}.\]
Then for the second expression, Lemma \ref{lem-calculus-kernel-holomorphic-extension} shows that, in fact 
\[ \limsup_{\epsi \to 0} \sup_{|t-t_0| < \epsi, \: t_0 \in [a,b]} |\exp(itA)g(x,\omega)- \exp(it_0 A)g(x,\omega)| = 0\]
for a.e. $(x,\omega) \in \Omega \times \Omega'$.
Thus the second expression vanishes.
We obtain that
\[ \norm{ \limsup_{\epsi \to 0} \sup_{|t-t_0| < \epsi, \: t_0 \in [a,b]} |\exp(itA)f - \exp(it_0A)f| \:}_{L^p(Y)} \lesssim \norm{(1+A)^\delta(f-g)}_{L^p(Y)} .\]
By Lemma \ref{lem-density-calculus-kernel-in-H}, $\norm{(1+A)^\delta(f-g)}_{L^p(Y)}$ is finite for any $g \in D_A$ and by density becomes arbitrarily small.
We conclude that for a.e. $(x,\omega) \in \Omega \times \Omega'$,
\[ \limsup_{\epsi \to 0}\sup_{|t-t_0| < \epsi, \: t_0 \in [a,b]} |\exp(itA)f(x,\omega) - \exp(it_0A)f(x,\omega)| = 0.\]
In other words, for a.e. $(x,\omega) \in \Omega \times \Omega'$, $[a,b] \to \C,\: t \mapsto \exp(itA)f(x,\omega)$ is continuous.
Now let $a \to 0+$ and $b \to \infty$ to see that $\R_+ \to \C, \: t \mapsto \exp(itA)f(x,\omega)$ is continuous.

For the second mapping, it suffices to redo the same proof starting with 
\[ \norm{\limsup_{\epsi \to 0} \sup_{|t-t_0| < \epsi, \: t_0 \in [a,b]}\norm{\exp(itA)f(x) - \exp(it_0A)f(x)| }_Y\: }_{L^p(\Omega)} \]
on the l.h.s. of the estimate, and already in the first step to use that this is majorised by $\norm{ \limsup_{\epsi \to 0} \sup_{|t-t_0| < \epsi, \: t_0 \in [a,b]}|\exp(itA)f(x) - \exp(it_0A)f(x)| \: }_{L^p(Y)}$.
\end{proof}

\begin{remark}
\label{rem-continuous-path-measurability}
In the situation of Theorem \ref{thm-continuous-paths}, one might wonder whether for general $f \in D(A^\delta)$, $\Omega \times \Omega' \to C[a,b],\: (x,\omega) \mapsto (t \mapsto \exp(itA)f(x,\omega))$ is a priori mesurable, in particular whether the scalar function $(x,\omega) \mapsto \sup_{t \in [a,b]} |\exp(itA)f(x,\omega)|$ is measurable.

To this end, note first that if $f$ belongs to $D_A$, then according to Lemma \ref{lem-calculus-kernel-holomorphic-extension}, $(t,x,\omega) \mapsto \exp(itA)f(x,\omega)$ is measurable and also 
\[ t \mapsto \sup_{t \in [a,b]}|\exp(itA)f(x,\omega)| = \sup_{t \in [a,b],\:t \in \mathbb Q} |\exp(itA)f(x,\omega)| \]
is measurable, where equality follows from continuity in $t$.
So, $\exp(itA)f$ then belongs indeed to $L^p(Y(C[a,b]))$.
Then for general $f \in D(A^\delta)$, we let $f_n \in D_A$ approximate $f$ in $H$ according to Lemma \ref{lem-density-calculus-kernel-in-H}.
Then according to the proof of Theorem \ref{thm-continuous-paths} as it stands,
\[\norm{t \mapsto \exp(itA)(f_n - f_m) }_{L^p(Y(C[a,b]))} \lesssim \norm{f_n - f_m}_H, \]
so that $(\exp(itA)f_n)_n$ is a Cauchy sequence in $L^p(Y(C[a,b]))$, and possesses a limit $X \in L^p(Y(C[a,b]))$.
Moreover, for fixed $t \in [a,b]$, 
\[\exp(itA)f_n = \left[(1+A)^{-\delta}\exp(itA)\right] (1 + A)^\delta f_n \to \exp(itA)f, \]
since $(1+A)^{-\delta}\exp(itA)$ is a bounded operator thanks to H\"ormander calculus and $(1+A)^\delta f_n \to (1+A)^\delta f$ in $L^p(Y)$, as $f \in D(A^\delta)$.
By unicity of limit, this implies that for all $t > 0$ and a.e. $(x,\omega) \in \Omega \times \Omega'$: $\exp(itA)f(x,\omega) = X(t,x,\omega)$.
Then we choose for each fixed $t > 0$ the representative of $\exp(itA)f$ such that it coincides with $X(t,\cdot,\cdot)$.
Thus we conclude that we can choose for any $t > 0$ a representative $\Omega \times \Omega' \to \C$ of $\exp(itA)f$ in $L^p(Y)$ such that for a.e. $(x,\omega) \in \Omega \times \Omega'$, the function $(0,\infty) \to \C,\: t \mapsto \exp(itA)f(x,\omega)$ is continuous.
\end{remark}

With the price that the initial value $f$ lies in a fractional domain space of $A$ of a higher order, we can obtain a.e. continuous paths $t\mapsto \exp(itA)f(x,\omega)$ on the whole real line and in particular continuity in $0$ as in the classical Carleson problem.

\begin{cor}
\label{cor-continuity-on-the-real-line}
Let $Y$ be a UMD lattice, $1 < p < \infty$ and $(\Omega,\mu)$ a $\sigma$-finite measure space.
Let $A = A_0 \ot \Id_Y$ have a $\Hor^\alpha_2$ calculus on $L^p(Y)$ for some $\alpha > \frac12$, and $A_0$ generate a lattice positive and contractive semigroup on $L^p(\Omega)$.
Pick
\[ \delta > 2\alpha + \max \left( \frac12, \frac{1}{\type L^p(Y)} - \frac{1}{\cotype L^p(Y)} \right) + \frac{3}{2} . \]
Then for any $f \in D(A^\delta)$, for a.e. $(x,\omega) \in \Omega \times \Omega'$, the function
\[ \begin{cases} \R & \longmapsto \C \\ t & \longmapsto \exp(itA)f(x,\omega) \end{cases} \]
is continuous.
Similarly, for a.e. $x \in\Omega$, the function
\[ \begin{cases} \R & \longmapsto Y \\ t & \longmapsto \exp(itA)f(x,\cdot) \end{cases} \]
is continuous.
\end{cor}

\begin{proof}
Fix some $t_0 > 0$ and $\epsi > 0$ such that $\delta > \epsi + 2\alpha + \max \left( \frac12, \frac{1}{\type L^p(Y)} - \frac{1}{\cotype L^p(Y)} \right) + \frac{3}{2}$.
Put $\delta' = \delta - \alpha - \epsi$ which then satisfies the assumption on $\delta$ in Theorem \ref{thm-continuous-paths}.
Let 
\[ g := \exp(-it_0 A)f = (1+A)^{-(\alpha + \epsi)} \exp(-it_0A) (1+A)^{\alpha + \epsi}f. \]
As $\lambda \mapsto (1+\lambda)^{-(\alpha + \epsi)} \exp(-it_0\lambda)$ belongs to $\Hor^\alpha_2$ \cite[Lemma 3.9 (2)]{KrW3} and $(1+A)^{\alpha + \epsi}f$ belongs to $D(A^{\delta'})$, $g$ belongs to $D(A^{\delta'})$ too.
Then we infer by Theorem \ref{thm-continuous-paths} that for a.e. $(x,\omega)$, $t \mapsto \exp(itA)g(x,\omega)$ is continous $(0,\infty) \to \C$ (resp. for a.e. $x$, $t \mapsto \exp(itA)g(x,\cdot)$ is continuous $(0,\infty) \to Y$).
However, by functional calculus, $\exp(itA)g = \exp(i(t-t_0)A)f$, so that $t \mapsto \exp(itA)f(x,\omega)$ is continuous $(-t_0,\infty) \to \C$ (resp. $t \mapsto \exp(itA)f(x,\cdot)$ is continuous $(-t_0,\infty) \to Y$) a.e.
We conclude by letting $t_0 \to \infty$.
\end{proof}

In case that $L^p(\Omega,Y) = L^2(\Omega,\C)$ and $A$ self-adjoint positive, Corollary \ref{cor-continuity-on-the-real-line} comes with (smaller) better exponent $\delta = 2 +\epsi$.
At this point, we remark that e.g. for the case of the free Schr\"odinger propagator and $f \in H^{\frac{d}{2} + \epsilon}(\R^d)$, the pointwise convergence $\exp(it\Delta)f(x) \to f(x)$ as $t \to 0+$ follows from a simple application of Fourier-Plancherel and the H\"older inequality.
Since in Proposition \ref{prop-continuity-on-the-real-line-selfadjoint} below, $\delta = 2 + \epsi$ does not depend on the dimension, this result does not seem to be reducible to such easy argument.

\begin{prop}
\label{prop-continuity-on-the-real-line-selfadjoint}
Let $(\Omega,\mu)$ be a $\sigma$-finite measure space.
Let $A$ be a self-adjoint positive operator on $L^2(\Omega)$ such that $e^{-tA}$ is positivity preserving for all $t \geq 0$ (more generally, such that $|e^{-tA}f| \leq S_t|f|$ where $S_t : L^2(\Omega) \to L^2(\Omega)$ is a positive contraction semigroup).
Let $\delta > 2$.
Then for any $f \in D(A^\delta)$, for a.e. $x \in \Omega$, the function
\[ \begin{cases} \R & \longmapsto \C \\ t & \longmapsto \exp(itA)f(x) \end{cases} \]
is continuous.
\end{prop}

\begin{proof}
The proof goes along the same lines as that of Theorem \ref{thm-main-ft}, Corollary \ref{cor-wave}, Theorem \ref{thm-continuous-paths} and Corollary \ref
{cor-continuity-on-the-real-line}.
This time, we pick $\alpha > \frac12$, since the selfadjoint calculus of $A$ implies $\Hor^\alpha_2$ calculus for such $\alpha$ (note that $\Hor^\alpha_2$ embeds into the space of bounded Borel functions on $[0,\infty)$).
Then the additional summand $\max \left( \frac12, \frac{1}{\type L^p(Y)} - \frac{1}{\cotype L^p(Y)} \right)$ in the choice of $\delta$ is not necessary, since $L^2(\Omega)$ is a Hilbert space, and thus, bounded $\Hor^\alpha_2$ calculus is equivalent to $R$-bounded $\Hor^\alpha_2$ calculus (see Proposition \ref{prop-Hormander-calculus-to-R-Hormander-calculus} and the proof of Theorem \ref{thm-main-ft}).
Also the additional summand $\alpha$ in Corollary \ref{cor-continuity-on-the-real-line} is not necessary, since $\exp(itA)$ is bounded on $L^2(\Omega)$ (without a smoothing power of resolvent of $A$).
We conclude that we can choose $\delta > \alpha + 0 + \frac32 = \frac12 + \frac32 = 2$.
\end{proof}

In the following let us compare Proposition \ref{prop-continuity-on-the-real-line-selfadjoint} with other results in the literature on continuity of $t \mapsto \exp(itA)f(x)$.
First note that in case $A = - \Delta$, the optimal parameter $s$ in order that this continuity in $t =  0$ holds for all $f \in H^s(\R^d)$ is Carleson's problem \cite[p.~24]{Car}, with solution $s > \frac{d}{2(d+1)}$.
A generalization to fractional powers $(-\Delta)^{a/2}$ of $-\Delta$ was obtained in \cite[Theorem 1.2]{Zhang}. See also \cite{ChoKo,MYZ} for related results in dimension two, and \cite{Sj2,SjS} for higher dimensions but sequential convergence $\exp(it_k(-\Delta)^{a/2})f(x) \to f(x)$ for sequences $(t_k)_k$ converging rapidly to $0+$.

\begin{prop}\cite[Theorem 1.2]{Zhang}
Let $a > 0$, $d \in \N$ and $s > d\min(1,a)/4$.
Let $A = (-\Delta)^{a/2}$.
Then $\exp(itA)f(x) \to f(x)$ as $t\to 0$ for a.e. $x \in \R^d$ if $f \in H^s(\R^d)$.
\end{prop}

An application of Proposition \ref{prop-continuity-on-the-real-line-selfadjoint} yields the following improvement in case of high dimension (e.g. $d > 8$ if $a \in (0,1)$).

\begin{cor}
\label{cor-Zhang}
Let $0 < a < 2$, $d \in \N$ and $s > 2a$.
Let $A = (-\Delta)^{a/2}$.
Then $\exp(itA)f(x) \to f(x)$ as $t \to 0$ for a.e. $x \in \R^d$ if $f \in H^s(\R^d)$.
\end{cor}

\begin{proof}
Note that we can apply Proposition \ref{prop-continuity-on-the-real-line-selfadjoint}.
Indeed, $A$ generates a positive contraction semigroup on $L^2(\R^d)$ (as it is subordinated to the heat semigroup) and is self-adjoint.
Our proposition applies for $f \in D(A^\delta) = D((-\Delta)^{a\delta/2}) = H^{a \delta}(\R^d)$ with $\delta > 2$ so with $a \delta > 2 a$.
\end{proof}

In another direction, in \cite[Theorem 1.4]{LW}, a.e. convergence of $\exp(itA)f(x)$ is proved with $A = P(D)$, where $P : \R^d \to \R$ is a real valued function,
and $P(D)$ is the selfadjoint Fourier multiplier with symbol $P(\xi)$.

\begin{prop}\cite[Theorem 1.4]{LW}
Let $m,s_0 > 0$ and $d \in \N$.
Let $P : \R^d\to \R$ be continuous such that $|P(\xi)| \lesssim |\xi|^m$ as $|\xi| \to \infty$.
Assume that for each $s > s_0$ the following maximal inequality holds.
\[ \norm{ \sup_{0 < t < 1} \left| e^{it P(D)} f \right| }_{L^p(B(0,1))} \lesssim \norm{f}_{H^s(\R^d)}, \quad p \geq 1 . \]
Then for all $f \in H^{s + \delta}(\R^d),\: 0 \leq \delta < m$,
\[ e^{itP(D)}f(x) - f(x) = o(t^{\frac{\delta}{m}}) \]
a.e. $x \in \R^d$ as $t \to 0+$.
\end{prop}

With somewhat different assumptions, we obtain pointwise convergence (without $o$ information) for such Fourier multipliers $A = P(D)$ in the following way.
Note that the lattice positivity assumption holds e.g. if $d = 1$ and $P$ is an even function such that its restriction to $[0,\infty)$ is concave.
For then, the Fourier transform of $e^{-tP(\xi)}$, which is the convolution kernel of $e^{-tP(D)}$, takes positive values.
The lattice positivity assumption also will hold if $P$ is negative definite, for then by Schoenberg's Theorem \cite{Schoe}, $e^{-tP(\xi)}$ is positive definite, so by Bochner's theorem \cite{Loo}, its Fourier transform is a positive measure. Since $\xi \mapsto |\xi|^a$ is negative definite for $0 < a \leq 2$, the corollary below applies e.g. if $P(\xi) = \sum_{k=1}^m \alpha_k |\xi|^{a_k}$ with $\alpha_k \geq 0$ and $a_k \in (0,2]$.

\begin{cor}
\label{cor-Li-Wang}
Let $P : \R^d \to \R_+$ be a measurable function taking positive values and $P(D)$ the associated positive self-adjoint Fourier multiplier.
Assume that $e^{-tP(D)}$ is lattice positive (more generally, $|e^{-tP(D)}f| \leq S_t |f|$ for some positive contraction semigroup $S_t : L^2(\R^d) \to L^2(\R^d)$).
Let $f \in D(P(D)^\delta)$ with $\delta > 2$ (e.g. if $|P(\xi)| \cong |\xi|^a$, then $f \in  H^{s}(\R^d)$ with $s > 2a$ suffices).
Then
\[ e^{itP(D)}f(x) \to f(x) \]
a.e. $x \in \R^d$ as $t \to 0+$.
\end{cor}

\begin{proof}
Apply Proposition \ref{prop-continuity-on-the-real-line-selfadjoint} directly.
\end{proof}

In a third direction, the Schr\"odinger equation and the Carleson problem can also be studied on domains $\Omega \subseteq \R^d$.
The following result has been obtained in \cite[(1.4)]{Zhe}.

\begin{prop}\cite[Corollary 1.2]{Zhe}
Let $\Omega$ be the right half plane $\{(x,y) \in \R^2 : \: x > 0 \}$.
Let $\Delta_D = \partial_x^2 + (1+x)\partial_y^2$ together with Dirichlet boundary conditions on $\partial \Omega$.
Let $s > \frac12$ and $f \in H^s_D(\Omega)$, where the latter space is the completion of $C^\infty_c(\Omega)$ with respect to the norm $\norm{f}_{H^s_D(\Omega)}^2 = \norm{f}_{L^2(\Omega)}^2 + \norm{(- \Delta_D)^{\frac{s}{2}} f}_{L^2(\Omega)}^2$.
Then $\exp(it\Delta_D)f(x,y) \to f(x,y)$ as $t \to 0+$ for a.e. $(x,y) \in \Omega$.
\end{prop}

We obtain the following variant of this result.
Note that we are able to take quite general open (not necessarily convex) domains $\Omega$ and second order differential operators with Dirichlet boundary conditions, but have to take the larger exponent $\delta > 2$, which however is still dimension free (and thus beyond any Sobolev embedding argument for high dimensions).

\begin{cor}
\label{cor-Zheng}
Let $\Omega \subseteq \R^d$ be open non-empty equipped with Lebesgue measure.
Let $A$ be the unbounded selfadjoint operator on $L^2(\Omega)$ associated to the form $a : H^1_0(\Omega) \times H^1_0(\Omega) \to \C$
\[ a(u,v) = \int_\Omega \sum_{j,k = 1}^d a_{kj} \partial_k u \overline{\partial_j v} dx, \]
where $a_{kj} = a_{jk} : \Omega \to \R$ are bounded measurable such that the principal part is elliptic.
That is, there exists a constant $\eta > 0$ such that
\[ \sum_{j,k = 1}^d a_{kj}(x) \xi_j \overline{\xi_k} \geq \eta |\xi|^2 \quad (\xi \in \C^n,\: \text{a.e.} \: x \in \Omega) . \]
Let $\delta > 2$.
Then for any $f \in D(A^\delta)$, 
\[ \exp(itA)f(x) \to f(x)\]
as $t \to 0+$ for a.e. $x \in \Omega$.
\end{cor}

\begin{proof}
According to \cite[Corollary 4.3 p. 104]{Ouh}, $e^{-tA}$ is lattice positive.
Moreover, according to \cite[Corollary 4.12 p. 115]{Ouh}, $e^{-tA}$ is $L^\infty$ contractive, so by self-adjointness $L^1$ contractive and by complex interpolation $L^2$ contractive.
Thus, $A$ is positive self-adjoint, so that it suffices to apply Proposition \ref{prop-continuity-on-the-real-line-selfadjoint}.
\end{proof}

Note that in the same source \cite{Ouh}, one can find more general second order operators e.g. with lower order terms, or with other boundary conditions such as Neumann, or mixed boundary conditions, that still satisfy the assumptions and thus the conclusion of Proposition \ref{prop-continuity-on-the-real-line-selfadjoint}.
In other words, for such operators, Corollary \ref{cor-Zheng} also holds.

\section{Examples and concluding remarks}
\label{sec-examples}

In this section, we shortly discuss examples of $0$-sectorial operators $A$ and of UMD lattices $Y$, to illustrate our main results from Sections \ref{sec-q-variation} and \ref{sec-wave}.
We will then end with some concluding remarks and open questions related to the present article.

\paragraph{H\"ormander calculus for (generalised) Gaussian estimates}

For the first result below, we refer to \cite[Definitions 4.2 and 4.3]{DK} for the definition of a semigroup with Gaussian and generalised Gaussian estimates.
We recall the main theorem on H\"ormander functional calculus on $L^p(Y)$ for Gaussian estimates from \cite{DKK}.

\begin{defi}
Let $p \in (1,\infty),$ $p_Y \in (1,2]$ and $q_Y \in [2,\infty).$
We put 
\begin{equation}
\label{equ-defi-alpha}
\alpha(p,p_Y,q_Y) = \max\biggl(\frac{1}{p},\frac{1}{p_Y},\frac12 \biggr) - \min \biggl(\frac{1}{p},\frac{1}{q_Y},\frac12 \biggr) \in (0,1).
\end{equation}
Informally spoken, this is the length of the segment, which is the convex hull of the points $\frac{1}{p},\frac{1}{p_Y},\frac{1}{q_Y}$ and $\frac12$ sitting on the real line.
\end{defi}

\begin{thm}{\cite[Theorem 4.10]{DKK}}
\label{thm-Hoermander}
Let $(\Omega,\dist,\mu)$ be a space of homogeneous type with a dimension $d$.
Let $A$ be a self-adjoint operator on $L^2(\Omega)$ generating the semigroup $(T_t)_{t \geq 0}$.
Let $p_0 \in [1,2)$ and $m \in [2,\infty)$.
Assume that $(T_t)_{t \geq 0}$ satisfies Gaussian (resp. generalised Gaussian) estimates with parameter $m$ (resp. $p_0,m$).
Let $Y$ be a UMD lattice which is $p_Y$-convex and $q_Y$-concave for some $p_Y \in (1,2]$ (resp. $p_Y \in (p_0,2]$) and $q_Y \in [2,\infty)$ (resp. $q_Y \in [2,p_0')$).
Assume that the convexifications $Y^{p_Y}$ and $(Y')^{q_Y'}$ are also UMD lattices.
For example, if $Y = L^s(\Omega')$ for some $s \in (1,\infty)$, then any $p_Y \in (p_0,s)$ and any $q_Y \in (s,p_0')$ are admissible.
Finally, assume that $A$ has a bounded $\HI(\Sigma_\omega)$ calculus on $L^p(\Omega,Y)$ for some fixed $p \in (1,\infty)$ (resp. $p \in (p_0,p_0')$) and $\omega \in (0,\pi)$.

Then $A$ has a H\"ormander $\Hor^\beta_2$ calculus on $L^p(\Omega,Y)$ with
\[ \beta > \alpha(p,p_Y,q_Y) \cdot d + \frac12\]
and $\alpha$ from \eqref{equ-defi-alpha}.
\end{thm}

Theorem \ref{thm-Hoermander} applies to a wide range of differential operators in different contexts, including those listed in the Introduction Section \ref{sec-Introduction}.
For a discussion of many further examples where Gaussian estimates are satisfied, we refer to \cite[Section 7]{DuOS}, see also \cite[Subsection 5.1]{DKK} and \cite[Subsection 4.1]{DK}.

Note also that our semigroups do not need to be self-adjoint on $L^2(\Omega)$ (if ever $p = 2$).
Indeed, suppose that the semigroup $(T_t)_{t \geq 0}$ consists of self-adjoint contractions on $L^2(\Omega,\mu)$, where $\Omega$ is a space of homogeneous type, with dimension $d \in \N$.
Then according to \cite[Theorem 3.2]{DSY}, there exists some $r_0 \in [1,2)$ such that for any $p \in (r_0,\infty)$, any $w \in A_{\frac{p}{r_0}}$ (Muckenhoupt weight class \cite[p.~1109]{DSY}) and $s > \frac{d+1}{2}$, the generator $A$ of $(T_t)_{t \geq 0}$ has a H\"ormander calculus on $L^p(\Omega,w d \mu)$.
Picking $p = 2$, the semigroup $(T_t)_{t \geq 0}$ is thus uniformly bounded on $L^2(\Omega,w d\mu)$, but if $w \in A_{\frac{2}{r_0}}$ is not constant, the semigroup $(T_t)_{t \geq 0}$ and the generator $A$ are not self-adjoint on this weighted $L^2$ space.

\paragraph{Examples of UMD lattices}

Apart from $L^q(\Omega')$ spaces for $1 < q < \infty$ which are the major examples of UMD lattices, one can also consider their intersections and sums when seen as subspaces of the common superspace of (equivalence classes of) measurable functions over $\Omega'$.

\begin{lemma}
\label{lem-intersection-sum-of-Lq}
Let $1 < q_1,q_2 < \infty$.
Let $\Omega'$ be a $\sigma$-finite measure space.
\begin{enumerate}
\item Then $Y = L^{q_1}(\Omega') \cap L^{q_2}(\Omega')$ equipped with the norm $\norm{f}_Y = \max\{\norm{f}_{q_1},\norm{f}_{q_2}\}$ and obvious partial order is a UMD lattice.
\item Then $Y = L^{q_1}(\Omega') + L^{q_2}(\Omega')$ equipped with the norm $\norm{f}_Y = \inf\{ \norm{g}_{q_1} + \norm{h}_{q_2} : \: f = g + h \}$ and obvious partial order is also a UMD lattice.
\end{enumerate}
\end{lemma}

For a simple proof, we refer to \cite[Lemma 4.9]{DK}.

\paragraph{Concluding remarks}

Already for classical (i.e. without $q$-variation or supremum in $t > 0$) H\"ormander multiplier theorems, a nice description of the exact norm $\|m(A)\|_{L^p \to L^p}$ in terms of a function norm $\|m\|$ of the spectral multiplier is not known today.
This problem is equally present for our $q$-variational and maximal spectral multiplier estimates in this article and only a step by step progression of sufficient conditions in the form $\|t \mapsto m(tA)(\cdot)\|_{L^p \to L^p(V^q)} \lesssim \|m\|_{\ldots}$ seems to be manageable.
In this direction, it would be interesting to know whether in the contexts of Theorem \ref{thm-q-variation} 2. ($q$-variation) and of Theorem \ref{thm-main-ft} 2. (maximal estimate with time-dependent $f_t$) of semigroup generators, one can relax the summation condition to
\[ \sum_{n \in \Z} \frac{\|m(2^n\cdot)\dyad_0\|_{W^c_2(\R)}}{1 + |n|} < \infty \]
as for maximal estimates in the euclidean case \cite{CGHS}, or with an additional factor $\log(|n|+2)$ as in \cite{Choi}.
As another possible relaxation, one can ask the question, whether
\[ \sum_{n \in \Z} \|m(2^n \cdot)\dyad_0\|_{W^c_q(\R)}^2 < \infty \]
is sufficient for our estimates, which is known to be true for the classical maximal estimate of the euclidean Laplacian \cite[(1.3)]{CGHS}.
For this, it would be interesting to know whether there are other approaches to $q$-variational and time dependent maximal estimates than ours, in case of the euclidean Laplacian or sub-Laplacians on Lie groups.
Also $q$-variational estimates for spectral multipliers that do not decay at $\infty$ are not well understood.
Already the scalar case $Y = \C$ would be interesting.

In the context of Section \ref{sec-q-variation-spherical}, it would be interesting to get information on the norm $C_{p,q,d,Y}$ in Theorem \ref{thm-spherical-means} and on $C_{p,q,Y}$ in Proposition \ref{prop-HHL-alternative}, depending on $p,q(,d)$ and $Y$.
Also one could try to find a proof of Proposition \ref{prop-HHL-alternative} using H\"ormander functional calculus that also covers small, so finally all, dimensions.

\section{Acknowledgments}

The authors acknowledge financial support through the research program ANR-18-CE40-0021 (project HASCON).
They also respectively acknowledge financial support through the research program ANR-18-CE40-0035  (project REPKA) and ANR-17-CE40-0021 (project FRONT).

\vspace{0.2cm}

\footnotesize{
\noindent Luc Deleaval \\
\noindent
Laboratoire d'Analyse et de  Math\'ematiques Appliqu\'ees (UMR 8050)\\
Universit\'e Gustave Eiffel\\
5, Boulevard Descartes, Champs sur Marne\\
77454 Marne la Vall\'ee Cedex 2\\
luc.deleaval@univ-eiffel.fr\hskip.3cm
}

\vspace{0.2cm}
\footnotesize{
\noindent Christoph Kriegler\\
\noindent
Universit\'e Clermont Auvergne,\\
CNRS,\\
LMBP,\\
F-63000 CLERMONT-FERRAND,\\
FRANCE \\
URL: \href{http://lmbp.uca.fr/~kriegler/indexenglish.html}{http://lmbp.uca.fr/{\raise.17ex\hbox{$\scriptstyle\sim$}}\hspace{-0.1cm} kriegler/indexenglish.html}\\
christoph.kriegler@uca.fr\hskip.3cm
}
\end{document}